\documentclass[12pt,reqno]{amsart}
\usepackage[top=1.25in,bottom=1.25in,left=1in,right=1in]{geometry}
\usepackage{amsmath,amssymb,amsthm,pinlabel}
\usepackage{ifthen}
\usepackage{enumitem}
\usepackage{mathtools}
\usepackage{varwidth}

\usepackage[all,cmtip]{xy}
\usepackage[final, % change to final
colorlinks=true,
pdftitle={Minimal volume entropy of free-by-cyclic groups and 2--dimensional right-angled Artin groups},
pdfauthor={Corey Bregman and Matt Clay}]{hyperref}

 \usepackage{tikz,tikz-cd,amscd}
\usetikzlibrary{decorations.pathreplacing,decorations.markings}
 \usetikzlibrary{shapes,shapes.arrows,positioning}
\tikzstyle directed=[postaction={decorate,decoration={markings,
    mark=at position .5 with {\arrow{stealth}}}}]
\tikzstyle rdirected=[postaction={decorate,decoration={markings,
    mark=at position .5 with {\arrow{stealth reversed}}}}]

\theoremstyle{plain}
\newtheorem{theorem}{Theorem}
\newtheorem{lemma}[theorem]{Lemma}
\newtheorem{proposition}[theorem]{Proposition}
\newtheorem{corollary}[theorem]{Corollary}
\newtheorem{claim}[theorem]{Claim}
\newtheorem*{claim*}{Claim}

\theoremstyle{definition}
\newtheorem{definition}[theorem]{Definition}

\newtheorem{example}[theorem]{Example}
\newtheorem{remark}[theorem]{Remark}

\numberwithin{theorem}{section}
\numberwithin{equation}{section}

\newcommand{\CC}{\mathbb{C}}

\newcommand{\ZZ}{\mathbb{Z}}

\newcommand{\calE}{\mathcal{E}}
\newcommand{\calF}{\mathcal{F}}
\newcommand{\calG}{\mathcal{G}}

\newcommand{\calN}{\mathcal{N}}

\newcommand{\calV}{\mathcal{V}}
\newcommand{\calX}{\mathcal{X}}

\newcommand{\tT}{{\widetilde T}}
\newcommand{\tX}{{\widetilde X}}
\newcommand{\tY}{{\widetilde Y}}

\newcommand{\udelta}{\underline{\delta}}
\newcommand{\uomega}{\omega}
    
\newcommand{\const}{10^6}
    
\newcommand{\abs}[1]{\left\lvert {#1} \right\rvert} 
\newcommand{\I}[1]{\langle #1 \rangle}

\newcommand{\from}{\colon\thinspace}

\newcommand{\norm}[1]{\left\| {#1} \right\|}

\newcommand{\PPD}[3]{% % second partial derivative
\ifthenelse{\equal{#2}{#3}}%
{\frac{\partial^{2} #1}{\partial #2^{2}}}%
{\frac{\partial^{2} #1}{\partial #2 \partial #3}}%
}

\newcommand{\detG}[1][1]{%
\ifthenelse{\equal{#1}{1}}%
{\operatorname{det}_{G}}%
{\operatorname{det}_{#1}}%
}

\newcommand{\param}%
	{{\mathchoice{\mkern1mu\mbox{\raise2.2pt\hbox{$\centerdot$}}\mkern1mu}%
	{\mkern1mu\mbox{\raise2.2pt\hbox{$\centerdot$}}\mkern1mu}%
	{\mkern1.5mu\centerdot\mkern1.5mu}{\mkern1.5mu\centerdot\mkern1.5mu}}}

\DeclareMathOperator{\Aut}{Aut}
\DeclareMathOperator{\Out}{Out}
\DeclareMathOperator{\diam}{diam}
\DeclareMathOperator{\gdim}{gd}

\DeclareMathOperator{\entropy}{ent}

\DeclareMathOperator{\vol}{vol}

\begin{document}

%%%%%%%%%%%%%%%%%%%%%%%%%%%%%%%%%%%%%%%%%%%%%%%%%%%%%%%%%%%%%%%%%%%%%%%%%%%%% 

\title[Minimal volume entropy of $F_n \rtimes_\Phi \ZZ$ and 2--dimensional RAAGs]{Minimal volume entropy of free-by-cyclic groups and 2--dimensional right-angled Artin groups}

\author[C.~Bregman]{Corey Bregman}
\address{Department of Mathematics\\
University of Southern Maine\\
Portland, ME 04103, USA\\}
\email{\href{mailto:corey.bregman@maine.edu}{corey.bregman@maine.edu}}

\author[M. Clay]{Matt Clay}
\address{Department of Mathematics\\
University of Arkansas\\
Fayetteville, AR 72701, USA\\}
\email{\href{mailto:mattclay@uark.edu }{mattclay@uark.edu}}

%\thanks{\tiny Thanks!}

\begin{abstract}
Let $G$ be a free-by-cyclic group or a 2--dimensional right-angled Artin group.  We provide an algebraic and a geometric characterization for when each aspherical simplicial complex with fundamental group isomorphic to $G$ has minimal volume entropy equal to $0$.  In the nonvanishing case, we provide a positive lower bound to the minimal volume entropy of an aspherical simplicial complex of minimal dimension for these two classes of groups.  Our results rely upon a criterion for the vanishing of the minimal volume entropy for 2--dimensional groups with uniform uniform exponential growth.  %Specifically, when $G$ is a free-by-cyclic group, the lower bound is $\frac{\log 3}{12 \cdot \const}$; when $G$ is a 2--dimensional right-angled Artin group this lower bound is $\frac{\log 3}{2 \cdot \const}$.  
This criterion is shown by analyzing the fiber $\pi_1$--growth collapse and non-collapsing assumptions of Babenko--Sabourau~\cite{un:BS}.   
\end{abstract}

\maketitle

%\tableofcontents

%%%%%%%%%%%%%%%%%%%%%%%%%%%%%%%%%%%%%%%%%%%%%%%%%%%%%%%%%%%%%%%%%%%%%%%%%%%%% 

\section{Introduction}

The \emph{volume entropy} of a finite simplicial complex $X$ equipped with a piecewise Riemannian metric $g$ is defined as
\[\entropy(X,g)=\lim_{t\rightarrow\infty}\frac{1}{t}\log \vol(B_{x_0}(t),\tilde{g})\]
where $B_{x_0}(t)$ is the ball of radius $t$ centered at some point $x_0$ in the universal cover $\widetilde{X}$ and $\tilde{g}$ is the pull-back metric on $\widetilde{X}$.  This limit always exists and does not depend on the choice of $x_0$. Initially defined as a Riemannian manifold invariant, the volume entropy measures the exponential growth rate of the volume of balls in the universal cover and is related to the growth of the fundamental group (\v{S}varc\cite{ar:Svarc55} and Milnor~\cite{ar:Milnor68}) and to the dynamics of the geodesic flow.  Specifically, in this context, Dinaburg showed that the volume entropy gives a lower bound on the topological entropy of the geodesic flow~\cite{ar:Dinaburg71}.  Manning further showed that if the sectional curvatures for the metric are all nonpositive, then the volume entropy equals the topological entropy of the geodesic flow~\cite{ar:Manning79}.    

In order to obtain a topological invariant of $X$, it is natural to optimize the volume entropy over all piecewise Riemannian metrics.  To get an invariant that is nondegenerate, we must take into account the effect of scaling the metric by a constant and counteract this by multiplying the volume entropy by an appropriate root of the volume.  This leads to the notion of minimal volume entropy, introduced by Gromov originally in the context of Riemannian manifolds~\cite{ar:Gromov82}.  To this end, we set
\[ \omega(X,g) = \entropy(X,g)\vol(X,g)^{1/\dim(X)}. \]
The \emph{minimal volume entropy} of a finite simplicial complex $X$ is defined by \[ \omega(X) = \inf_g \omega(X,g) \] where $g$ runs over all piecewise Riemannian metrics on $X$.   

When $M$ is a closed, orientable $n$--manifold, Gromov showed that $\omega(M)^n\geq c_n\norm{M}$ where $\norm{M}$ is the simplicial volume of $M$ and $c_n>0$ is a constant that only depends on the dimension \cite{ar:Gromov82}.  In dimensions at most 3, the invariants $\omega(M)^n$ and $\norm{M}$ are proportional, as we explain below.  It is unknown whether or not the reverse inequality holds up to a constant in higher dimensions.  Nevertheless, it is in this sense that $\omega(X)^{\dim(X)}$ can be viewed as a substitute for simplicial volume for $X$ when there is no natural choice of fundamental class.

%Thus, $\omega(M)$ is nonzero for many manifolds, but in only a handful of cases has the exact value of $\omega(X)$ been computed.   

Katok was the first to realize that minimal volume entropy could select an optimal metric, up to scale.  He proved that if $M$ is a closed surface with negative Euler characteristic then $\omega(M,g) \geq \omega(M,g_{\rm hyp})$ where $g_{\rm hyp}$ is any hyperbolic metric, with equality if and only if $g$ has constant curvature~\cite{ar:Katok82}.  This was extended by Besson--Courtois--Gallot to closed, real hyperbolic manifolds of any dimension.~\cite{ar:BCG91}.

For simplicial complexes that are not manifolds, there are few results.  When $X$ is a finite connected graph and every vertex has degree at least 3, Lim gave an explicit description of a metric $g_0$ so that $\omega(X) = \omega(X,g_0)$~\cite{ar:Lim08}.  Analogous to the results for closed real hyperbolic manifolds mentioned above, Lim additionally proves that this metric is unique up to scale.  McMullen gave an alternate proof of this result~\cite{ar:McMullen15}; I.~Kapovich--Nagnibeda gave a proof of this result when every vertex in the graph has degree 3~\cite{ar:KN07}.  

Other general results regarding minimal volume entropy for simplicial complexes include the fiber $\pi_1$--growth collapsing/non-collapsing assumptions recently provided by Babenko--Sabourau that are useful in showing whether or not $\omega(X)$ vanishes~\cite{un:BS}.  These will play a key role in this paper and are discussed in more detail later on in the Introduction and in Section~\ref{sec:collapsing}.

As mentioned above, the volume entropy is related to the growth of the fundamental group in that it---or a slight variation---can be used to determine the growth type: polynomial or exponential.  However, in general, the minimal volume entropy of a simplicial complex does depend on more than the fundamental group, as originally observed by Babenko~\cite{ar:Babenko92}.  (Although it will not play a role in what follows, in the context of manifolds there are circumstances where the minimal volume entropy is determined by the fundamental group; see the works of Babenko~\cite{ar:Babenko92} and Brunnbauer~\cite{ar:Brunnbauer08}.)

This leads into the central object of study in this paper.  For a fixed group $G$ we study the minimal volume entropy of a $G$--complex i.e., a finite aspherical simplicial complex $X$ such that $\pi_1(X) \cong G$.    By taking the infimum over $G$--complexes with minimal dimension, we obtain an invariant of a group $G$ of finite type.  We thus define the \emph{minimal volume entropy of $G$} as 
\[ \uomega(G)  = \inf_X \omega(X) \] 
where $X$ runs over all $G$--complexes with $\dim(X)$ equal to the geometric dimension, $\gdim(G)$, i.e., the minimal dimension of a $G$--complex.  For free groups, it was observed by both I.~Kapovich--Nagnibeda~\cite{ar:KN07} and McMullen~\cite{ar:McMullen15} that if $X$ is a finite graph and $\pi_1(X)$ is isomorphic to a free group of rank $n$, then $\omega(X) \geq (3n-3)\log 2$ with equality if and only if every vertex in $X$ has degree 3.  Cast in the above language, this gives $\uomega(F_n) = (3n-3)\log 2$. 

%Thus, $\uomega(G)$ can be viewed as an analogue of the simplicial volume even in the absence of a fundamental class.  

%free-by-cyclic groups

In this paper we study the minimal volume entropy when $G$ is either a free-by-cyclic group or a 2--dimensional right-angled Artin group (RAAG).  Each such group $G$ admits a 2--dimensional aspherical $G$--complex.  In each case, we prove that either the minimal volume entropy vanishes for every $G$--complex or $\uomega(G)$ is uniformly bounded from below. Moreover, as we will describe below, whether or not $\uomega(G)$ vanishes is directly related to whether or not $G$ is \emph{tubular}, i.e., whether it admits a graph of groups decomposition with vertex groups equal to $\ZZ^2$ and edge groups equal to $\ZZ$. 

We state our results in these two cases and then describe how we apply the fiber $\pi_1$--growth collapsing/non-collapsing assumptions of Babenko--Sabourau.

\medskip \noindent {\it Free-by-cyclic groups.} Every free-by-cyclic group is determined by a finite rank free group $F_n$ and an element $\phi\in \Out(F_n)$, the outer automorphism group of $F_n$. Denote by $G_\phi$ the free-by-cyclic group associated to $\phi$.  Specifically, the group $G_\phi$ is given by the presentation:
\[ G_\phi = \I{F_n,t \mid txt^{-1} = \Phi(x)} \]
where $\Phi \in \Aut(F_n)$ represents $\phi$.

We are able give an explicit description---up to passing to a power---for which $\phi$ lead to vanishing minimal volume entropy.   We call outer automorphisms with such a description \emph{geometrically linear unipotent} (GLU) (see Definition \ref{def:tubular aut}). We prove:

%PG, EG, atoroidal

\begin{theorem}\label{thm:FbyZ}
Suppose that $\phi$ is an outer automorphism of a finitely generated free group. The following are equivalent:
\begin{enumerate}
\item  $\omega(X) = 0$ for every $G_\phi$--complex $X$.\label{item:FbyZ omega = 0}
\item $G_\phi$ is virtually tubular.\label{item:FbyZ tubular}  
\item Some power of $\phi$ is geometrically linear unipotent.\label{item:FbyZ glu}
\end{enumerate}
If none of these conditions hold, then $\uomega(G_\phi) \geq \frac{\log 3}{12 \cdot \const}$.
\end{theorem}

There is an established connection between free-by-cyclic groups and mapping tori $M_f$ of homeomorphisms of closed orientable surfaces $f \from S \to S$.  GLU automorphisms have linear growth, and via this connection are reminiscent of a multi-twist homemorphism of a closed orientable surface.  Pieroni showed that $\omega(M)^3$, for $M$ a closed orientable 3--manifold, equals two times the sum of the volumes of the hyperbolic components in the JSJ decomposition~\cite{un:Pieroni}.  In particular, by the celebrated result of Thurston, the minimal volume entropy of $M_f$ vanishes if and only if some power of $f$ is homotopic to a multi-twist~\cite{ar:Thurston82}. In this way, Theorem~\ref{thm:FbyZ} result can be regarded as a free group analogue.  However, in contrast to the case of a mapping class on a closed surface, not all subexponentially growing outer automorphisms of free groups have linear growth, and not all $G_\phi$ with linearly growing $\phi$ have vanishing minimal volume entropy. 

The $L^2$--torsion $-\rho^{(2)}(\param)$ is an analytic invariant of certain groups that may also play the role of volume.  Indeed, if $M$ is a closed orientable 3--manifold, then L\"uck--Schick proved that $-\rho^{(2)}(\pi_1(M))$ equals $\frac{1}{6\pi}$ times the sum of the volumes of the hyperbolic components in the JSJ decomposition~\cite{ar:LS99}.  By combining the work of Gromov, Soma and Thurston, we have that $\norm{M}$ also equals a constant times the sum of the of the volumes of the hyperbolic components in the JSJ decomposition in this case~\cite{ar:Gromov82,ar:Soma81,un:Thurston78}.  Thus we see that these three notions of volume---minimal volume entropy, $L^2$--torsion, and simplicial volume---are all proportional for closed orientable 3--manifolds, in particular, for mapping tori of homeomorphisms of closed orientable surfaces.

As there is no well-defined fundamental class for a free-by-cyclic group, there is no natural way to define the simplicial volume.  However, it is interesting to compare the minimal volume entropy and the $L^2$--torsion for free-by-cyclic groups.  The second author proved that $-\rho^{(2)}(G_\phi)$ vanishes when $\phi$ is polynomially growing~\cite{ar:Clay17-2}---conjecturally, the converse holds as well.  As Theorem~\ref{thm:FbyZ} shows that most free-by-cyclic groups with polynomially growing monodromy have nonvanishing minimal volume entropy, we see that these two invariants are not proportional in this setting.  The second author provided an upper bound on $-\rho^{(2)}(G_\phi)$ using the dynamics of $\phi$~\cite{ar:Clay17-2}, it would be interesting to find an upper bound on the minimal volume entropy of $G_\phi$ as well.

Theorem~\ref{thm:FbyZ} provides a characterization of free-by-cyclic groups that are virtually tubular.  We note that Button has provided a characterization of tubular groups that are free-by-cyclic~\cite{ar:Button17}.

\medskip \noindent {\it Right-angled Artin groups.}  Let $\Gamma$ be a finite simplicial graph.  The right-angled Artin group $A_\Gamma$ is the group whose generators are the vertices of $\Gamma$ and whose relations are commutations between generators when the vertices are incident on an edge in $\Gamma$.  That is, $A_\Gamma$ is given be the presentation:
\[ A_\Gamma = \I{ V\Gamma \mid vw = wv \text{ if $v$ and $w$ are incident on an edge in $\Gamma$ } }. \]
Right-angled Artin groups, though simple to define, form an essential class of groups in low-dimensional topology and geometric group theory.  Partly, this is due to the suprising richness of their subgroups, their role as an interpolation between free groups and free abelian groups and also the frequency at which they arise as subgroups of geometrically defined groups.

The group $A_\Gamma$ has geometric dimension equal to 2 if and only if $\Gamma$ has no triangles, i.e., $K_3$ is not a subgraph of $\Gamma$.  In this case an $A_\Gamma$--complex, known as the Salvetti complex $S_\Gamma$, is built out of unions of circles $S^1$ and tori $S^1 \times S^1$ that are identified along certain cyclic subgroups.  This structure seems reminiscent of a tubular group, however, not all 2--dimensional right-angled Artin groups are (free products of) tubular groups.  In fact, whether or not a 2--dimensional right-angeled Artin group is tubular is directly related to its minimal volume entropy.  
\begin{theorem}\label{thm:raags}
Suppose that $A_\Gamma$ is a right-angled Artin group with $\gdim(A_\Gamma) = 2$.  The following are equivalent.
\begin{enumerate}
\item $\omega(X) = 0$ for every $A_\Gamma$--complex $X$.\label{item:raag omega = 0}
\item $A_\Gamma$ is a free product of tubular groups and a free group.\label{item:raag tubular}
\item $\Gamma$ is a forest.\label{item:raag forest}
\end{enumerate}  
If none of these conditions hold, then $\uomega(A_\Gamma) \geq \frac{\log 3}{2 \cdot \const}$.
\end{theorem}
We remark that according to Droms, $A_\Gamma$ is a 3-manifold group exactly when $\Gamma$ is a disjoint union of trees and triangles \cite{ar:Droms87}. Since a triangle corresponds to $\ZZ^3$, $\Gamma$ is a forest exactly when $A_\Gamma$ is a 3--manifold group with geometric dimension at most $2$.  

It seems likely a characterization for the vanishing of minimal volume entropy of  right-angled Artin groups of arbitrary dimension is possible, although the statement may not be so neat.  

\medskip \noindent {\it  Fiber $\pi_1$--growth assumptions.} Both theorems above are consequences of the fiber $\pi_1$--growth collapsing/non-collapsing assumptions of Babenko--Sabourau \cite{un:BS}. These assumptions relate the vanishing of the minimal volume entropy of a simplicial complex $X$ to the existence or non-existence of maps $f \from X \to P$ to lower dimensional complexes, based on the $\pi_1$--growth of fibers.  Briefly, the two assumptions are:
\begin{itemize}
\item Fiber $\pi_1$--Growth Collapsing Assumption (FCA) -- For some simplicial map $f \from X \to P$, every induced subgroup $\pi_1(f^{-1}(x)) \subseteq \pi_1(X)$ is subexponentially growing with subexponential growth rate less than $1 - \frac{\dim P}{\dim X}$.
\item Fiber $\pi_1$--Growth Non-Collapsing Assumption (FNCA) -- There is a constant $\delta$ such that for every simplicial map $f \from X \to P$, some induced subgroup $\pi_1(f^{-1}(x)) \subseteq \pi_1(X)$ is has uniform exponential growth rate at least $\delta$.
\end{itemize}

Although the two criteria are not \emph{a priori} complementary, we show that they are in the case of free-by-cyclic groups and 2--dimensional right-angled Artin groups.   Moreover, the two assumptions are complementary when the fundamental group of $X$ has uniform uniform exponential growth and satisfies a technical condition on subexponentially growing subgroups (see Definition~\ref{def:property U} and Proposition~\ref{lem:complementary}).

When $X$ is a 2--dimensional simplicial complex, the fiber $\pi_1$--growth assumptions consider maps $f \from X \to P$ where $P$ is a finite simplicial graph.  Applying standard geometric group theoretic techniques, when $X$ satisfies the FCA there is an induced graph of groups decomposition on $\pi_1(X)$ where the vertex and edge groups are all subexponentially growing (Proposition~\ref{prop:fca -> graph}). 

When a group $G$ has uniform uniform exponential growth (denoted by $\udelta(G) > 0$ in the following), we prove the following vanishing criteria:

\begin{theorem}\label{thm:zero iff graph of groups}
Let $G$ be a group with $\gdim(G) = 2$. Suppose $\udelta(G)>0$ and that the subexponentially growing subgroups of $G$ belong to the collection $\{\{1\}, \ZZ, \ZZ
^2, BS(1,-1) \}$.  Then $\omega(X) = 0$ for every $G$--complex $X$ if and only if $G$ is the fundamental group of a graph of groups where the edge groups belong to the collection $\{ \{ 1 \},\ZZ \}$ and the vertex groups belong to the collection $\{\ZZ, \ZZ^2, BS(1,-1) \}$. 
\end{theorem}

\begin{remark}\label{rem:zero divisor conjecture}
The Baumslag--Solitar group $BS(1,-1) = \I{a,t \mid tat^{-1} = a^{-1}}$ is the fundamental group of the Klein bottle.  By a result of Degrijse, if $G$ has cohomological dimension equal to 2, has subexponential growth and the group algebra $\CC[G]$ does not have zero-divisors, then $G$ is either $\ZZ^2$ or $BS(1,-1)$~\cite[Theorem~B]{un:Degrijse}.  Conjecturally, if $G$ is torsion-free then $\CC[G]$ does not contain any zero-divisors, which would render this hypothesis unnecessary.  Therefore conjecturally, Theorem~\ref{thm:zero iff graph of groups} applies to any group with $\gdim(G) = 2$ and $\udelta(G) > 0$.
\end{remark}

Theorem~\ref{thm:zero iff graph of groups} applies to the case of groups $G$ with $\gdim(G) = 2$ that act freely and cocompactly on CAT(0) cube complexes with isolated flats by recent work of Gupta--Jankiewicz--Ng~\cite{un:GJN}.

\subsection{Outline of paper} In Section \ref{sec:entropy and volume}, we discuss notions of growth in groups and show that the vanishing of $\omega(X)$ is a homotopy invariant of $G$--complexes of minimal dimension.  Section \ref{sec:collapsing} recalls the fiber $\pi_1$--growth collapsing/non-collapsing assumptions of Babenko--Sabourau~\cite{un:BS} and proves that these are complementary when $G$ has Property $U$.  After briefly reviewing graphs of groups, in Section \ref{sec:vanishing criterion} we prove Theorem~\ref{thm:zero iff graph of groups}% and prove Theorem~\ref{thm:cat 0} as an application
. In Section~\ref{sec:FbyZ} we prove Theorem~\ref{thm:FbyZ} regarding the minimal volume entropy of free-by-cyclic groups.  Finally, in Section~\ref{sec:raags}, we prove Theorem~\ref{thm:raags} regarding the minimal volume entropy of 2--dimensional right-angled Artin groups. 

\subsection{Acknowledgements} 

The authors would like to thank Derrick Wigglesworth for discussions regarding Lemma~\ref{lem:uniform uniform FbyZ}.  We also thank the Babenko and Sabourau and also the referee for pointing out an error in a previous version of the proof of Theorem~\ref{thm:FCA no volume}.  The first author is supported by NSF grants No.~DMS-1906269 and~DMS-2052801.  The second author is supported by Simons Foundation Grant No.~316383.

%%%%%%%%%%%%%%%%%%%%%%%%%%%%%%%%%%%%%%%%%%%%%%%%%%%%%%%%%%%%%%%%%%%%%%%%%%%%%

\section{Entropy and volume in groups}\label{sec:entropy and volume}

In this section discuss growth in groups and the relation between the minimal volume entropy of a group and the minimal volume entropy of a finite index subgroup.

\subsection{Growth in groups}\label{subsec:growth}

Let $G$ be a finitely generated group and suppose that $S \subset G$ is a finite generating set.  For an element $h \in G$, by $\norm{h}_S$ we denote the word length of $h$ with respect to $S$.  

The \emph{growth rate of $G$ with respect to $S$} is the quantity
\[ \delta(G,S) = \lim_{t \to \infty} \frac{1}{t} \log \#\{ h \in G \mid \norm{h}_S \leq t  \}.  \]  We observe that if $X$ is the Cayley graph of $G$ with respect to the generating set $S$ and $g$ is the piecewise Riemannian metric on $X$ for which each edge of $X$ is isometric to the unit interval, then $\entropy(X,g) = \delta(G,S)$.  If $\delta(G,S) > 0$ for some finite generating set $S \subset G$, then it is known that $\delta(G,S') > 0$ for all finite generating sets $S' \subset G$.  In this case, the group $G$ is said to have \emph{exponential growth}.  Else, the group is said to have \emph{subexponential growth}.  

In the case of subexponential growth, we consider the \emph{subexponential growth rate of $G$} which is defined by
\[ \nu(G) = \lim_{t \to \infty} \frac{\log \log \#\{ h \in G \mid \norm{h}_S \leq t  \}}{\log t}.  \]
This quantity satsifies $0 \leq \nu(G) \leq 1$.  We remark that if $\#\{ h \in G \mid \norm{h}_S \leq t  \}$ is bounded by a polynomial, then $\nu(G) = 0$.

In the case of exponential growth, to get a quantity that is independent of the generating set, we can take the infimum.  This leads to the \emph{uniform growth rate of $G$} which is defined by
\[ \delta(G) = \inf_S \delta(G,S) \] where $S$ runs over all finite generating sets for $S$.  The group $G$ is said to have \emph{uniform exponential growth} if $\delta(G) > 0$.  There are examples of finitely generated groups with exponential growth, but not uniform exponential growth~\cite{ar:Wilson04}.

Taking this concept one step further, we can take the infimum over all finitely generated exponentially growing subgroups of $G$ as well.  This leads to the \emph{uniform uniform growth rate of $G$} which is defined by
\[ \udelta(G)  = \inf_H \delta(H) \] where $H$ runs over all finitely generated subgroups of $G$ with exponential growth.  The group $G$ is said to have \emph{uniform uniform exponential growth} if $\udelta(G) > 0$.  

Finally, the following property is relevant to the sequel.

\begin{definition}\label{def:property U}
A finitely generated group $G$ has \emph{Property U} if $\udelta(G) > 0$ and if $H$ is a subgroup of $G$ with subexponential growth, then $\nu(H) = 0$.
\end{definition}

In particular, if $\udelta(G) > 0$ and every subexponentially growing subgroup of $G$ has polynomial growth, then $G$ has Property $U$.

\subsection{Monotone Maps and Finite Index Subgroups}\label{subsec:mve}

Let $X$ and $Y$ be simplicial complexes with $\dim(X) = \dim(Y) = m$.  A simplicial map $f \from Y \to X$ is said to be \emph{$n$--monotone} for $n\geq 0$ if the preimage of any open $m$--simplex in $X$ consists of at most $n$ open $m$--simplices in $Y$.  The following lemma gives a relation between the minimal volume entropy of $X$ and $Y$ using a $n$--monotone map $f \from Y \to X$.  This lemma appears in a paper by Brunnbauer~\cite[Lemma~4.1]{ar:Brunnbauer08}.  The proof is attributed to Babenko~\cite{ar:Babenko92} and appears in a paper by  Sabourau~\cite[Lemma~3.5]{ar:Sabourau06}. 

\begin{lemma}\label{lem:comparison}
Let $f \from Y \to X$ be an $n$--monotone map between $m$--dimensional finite simplicial complexes. If $f_* \from \pi_1(Y) \to \pi_1(X)$ is injective, then \[ n^{1/m} \cdot \omega(X) \geq \omega(Y).\] 
\end{lemma}

There are two useful consequences of this bound.

\begin{proposition}\label{prop:essential he invariant}
Let $G$ be a group of finite type and let $X$ and $Y$ be $G$--complexes with $\dim(X) = \dim(Y)$.  Then $\omega(X) = 0$ if and only if $\omega(Y) = 0$. 
\end{proposition}

\begin{proof}
Let $f \from X \to Y$ be a homotopy equivalence and let $f' \from Y \to X$ be the homotopy inverse to $f$. By the simplicial approximation theorem, we may assume that both of these maps are simplicial.  Let $m$ denote the common dimension of $X$ and $Y$.  By finiteness, each $m$--simplex of $Y$ has at most $n$ preimages for some $n>0$ under $f$. Similarly, there exists $n'>0$ such that each $m$--simplex of $X$ has at most $n'$ preimages under $f'$.  The proposition now follows from Lemma \ref{lem:comparison}, since $n$ and $n'$ are both positive.
\end{proof}

We record the following corollary of Proposition~\ref{prop:essential he invariant}.

\begin{corollary}\label{co:mve vanishes}
Let $G$ be a group of finite type and suppose that $\omega(X) = 0$ for some $G$--complex with $\dim(X) = \gdim(G)$.  Then $\omega(X) = 0$ for every $G$--complex $X$ with $\dim(X) = \gdim(G)$. 
\end{corollary}

The other useful consequence is with regards to finite index subgroups.

\begin{proposition}\label{prop:finite index}
Suppose that $G$ is a group of finite type.  If $H$ is a subgroup of $G$ and $[G:H] = n$, then \[n^{1/\gdim(G)} \cdot \uomega(G) \geq \uomega(H).\]
\end{proposition}

\begin{proof}
Let $G$ and $H$ be as in the statement.  As $G$ has finite type, so does $H$ and moreover $\gdim(G) = \gdim(H)$.

Suppose that $X$ is a $G$--complex with $\dim(X) = \gdim(G)$.  Let $f \from Y \to X$ be the cover corresponding to the subgroup $H$.  Then $f$ is $n$--monotone and $f_* \from \pi_1(Y) \to \pi_1(X)$ is injective.  Hence by Lemma~\ref{lem:comparison} we have that
\[ n^{1/\gdim(G)} \cdot \omega(X) \geq \omega(Y).\] 
As $X$ is an arbitrary $G$--complex with $\dim(X) = \gdim(G)$ and $\omega(Y) \geq \uomega(H)$ for any $H$--complex $Y$ with $\dim(Y) = \gdim(H)$, the result follows.
\end{proof}

In particular, if $\uomega(H) > 0$ for some finite index of a group $G$ of finite type, then $\uomega(G) > 0$ as well.

%%%%%%%%%%%%%%%%%%%%%%%%%%%%%%%%%%%%%%%%%%%%%%%%%%%%%%%%%%%%%%%%%%%%%%%%%%%%%

\section{Fiber \texorpdfstring{$\pi_1$--growth}{pi\_1-growth}  assumptions}\label{sec:collapsing}

In this section we recall the fiber $\pi_1$--growth collapsing and non-collapsing assumptions, introduced by Babenko--Sabourau~\cite{un:BS}. The collapsing assumption provides a sufficient condition for the minimal volume entropy to vanish, while the non-collapsing assumption guarantees it is nonzero, and also provides a lower bound. 

\subsection{Fiber \texorpdfstring{$\pi_1$--growth}{pi\_1-growth} collapsing assumption}

First we discuss the collapsing assumption.  Let $X$ be a simplicial complex.  A closed subset $F \subseteq X$ has \emph{subexponential growth} if for every connected component $F_0\subseteq F$, the inclusion induced image of $\pi_1(F_0)$ in $\pi_1(X)$ has subexponential growth.  

\begin{definition}[Babenko--Sabourau \cite{un:BS}]\label{def:FCA} 
A simplicial complex $X$ of dimension $m$ satisfies the \emph{fiber $\pi_1$--growth collapsing assumption} (FCA) if there exists a simplicial map $f \from X \to P$ to a finite simplicial complex $P$ of dimension $k$ such that for every $p \in P$, the fiber $f^{-1}(p)$ has subexponential growth with subexponential growth rate less than $\frac{m-k}{m}$.
\end{definition}

Babenko--Sabourau prove that the FCA is sufficient to ensure that the minimal volume entropy vanishes.  We will provide a proof of a weaker version that is sufficient for our needs.  Namely, we will show that satisfying the FCA with the stronger assumption that the subexponential growth rate of the fibers is less than $1/\dim X$ implies that the minimal volume entropy vanishes.  Our proof is based on their outline but the assumption about the subexponential growth rate of the fibers simplifies the argument.  

To do this, we need some facts about subexponential functions.  Let $\phi \from [0,\infty) \to [0,\infty)$ be a continuous, non-decreasing, subexponential function. By definition, for every $0 < \lambda \leq 1$, there exists $T_\lambda \in [0,\infty)$ such that for all $t\geq T_\lambda$
\[\phi(t)\leq \exp(\lambda t).\] We may assume $T_\lambda$ is the largest $t$ such that $\phi(t)=\exp(\lambda t)$. The next lemma describes the dependence of $T_\lambda$ on $\lambda$ as $\lambda\to 0^+$.

\begin{lemma}\label{lem:MovingZero}
Suppose that $\phi \from [0,\infty) \to [0,\infty)$ is a continuous, non-decreasing, subexponential function and that $\phi(t_0) > 1$ for some $t_0 \in (0,\infty)$.  The following statements hold.
\begin{enumerate}
\item The function $\lambda \mapsto T_\lambda$ is continuous and strictly decreasing on $(0,\lambda_0]$ where $\lambda_0 = \min\{1,\frac{1}{t_0}\log \phi(t_0)\}$.\label{item:1}
\item $\lim_{\lambda \to 0^+} T_\lambda = \infty$.\label{item:2}
\item The inverse function $t \mapsto \Lambda_t$ on $[t_0,\infty)$ defined by $T_{\Lambda_t} = t$ satisfies
\[ \limsup_{t \to \infty} \frac{\Lambda_t}{t^{\nu+\epsilon-1}} \leq 1 \] for any $\epsilon > 0$ where $\nu$ is the subexponential growth rate of $\phi(t)$.\label{item:3}
\end{enumerate}
\end{lemma}

\begin{proof}
Consider $h_\lambda(t)=\exp(\lambda t)-\phi(t)$, so that $T_\lambda$ is the largest $t$ for which $h_\lambda(t) = 0$.  Notice that $h_\lambda(t) \geq 0$ for $t \geq T_\lambda$ as well by definition.   Clearly $h_\lambda(t)$ is continuous in both $\lambda$ and $t$.  This implies that $T_\lambda$ is continuous as a function of $\lambda$.  

If $0 < \lambda \leq \lambda_0$, then $\exp(\lambda t_0) \leq \exp(\lambda_0 t_0) \leq \phi(t_0)$ which implies that $T_\lambda \geq t_0 > 0$.  Hence, for $0 < \lambda < \mu \leq \lambda_0$, as $\exp(\lambda t)<\exp(\mu t)$ for $t > 0$, we find that $h_\lambda(T_\mu) < h_\mu(T_\mu) = 0$.  Thus $T_\lambda > T_\mu$.  This shows that $T_\lambda$ is strictly decreasing on $(0,\lambda_0]$, equivalently, strictly increasing as $\lambda \to 0^+$.  This completes the proof of (1).

Suppose $T_\lambda \to T_0<\infty$ as $\lambda\to 0^+$. Since $T_\lambda$ is strictly decreasing on $(0,\lambda_0]$, we have that $T_0 > T_{\lambda_0} \geq t_0$.  Therefore for all $\lambda>0$, we have $h_\lambda(T_0)>0$ and hence $\lim_{\lambda\to 0^+}h_\lambda(T_0)\geq 0$.  On the other hand, for any fixed $t>0$, we have that $\exp(\lambda t)\to 1$ as $\lambda \to 0^+$. Thus, in particular, $\lim_{\lambda \to 0^+}\exp(\lambda T_0)=1$.  But then as $\phi(T_0)>1$, since $t_0 < T_0$ and $\phi$ is increasing, this implies $\lim_{\lambda\to 0^+}h_\lambda(T_0)<0$, a contradiction.  Therefore we conclude that $\lim_{\lambda \to 0^+} T_\lambda = \infty$.  This shows (2).

Finally, suppose that $0 \leq \nu \leq 1$ is the subexponential growth rate of $\phi(t)$.  Thus for any $\epsilon >0$ we have that $\phi(t) \leq \exp(t^{\nu+\epsilon})$ for large $t$.  Hence for large $t$, we have $\exp(\Lambda_t t) = \phi(t) \leq \exp(t^{\nu+\epsilon})$.  This gives $\Lambda_t \leq t^{\nu + \epsilon - 1}$ for large enough $t$ and thus $\limsup_{t \to \infty} \frac{\Lambda_t}{t^{\nu+\epsilon-1}} \leq 1$.  This shows (3).
\end{proof}

Given a simplicial map $f \from X \to P$, we will call an edge $e$ of $X$ \emph{long} if $f(e)$ is an edge of $P$, and \emph{short} otherwise, in which case $f(e)$ is a vertex of $P$.   
\begin{theorem}[{Babenko--Sabourau~\cite[Theorem~2.6]{un:BS}}]\label{thm:FCA no volume}
Let $X$ be a finite, connected, simplicial complex.  If $X$ satisfies the FCA, then $\omega(X) = 0$.
\end{theorem}

\begin{proof}[Proof when subexponential growth rate of fibers is less than $1/\dim X$]
Let $m$ denote the dimension of $X$.  Suppose $f \from X \to P$ is a simplicial map where $\dim P < m$ and for every $p \in P$, the fiber $f^{-1}(p)$ has subexponential growth with subexponential growth rate $\nu$ where $\nu < 1/m$.  Without loss of generality, we may assume that $f \from X \to P$ is surjective and has connected fibers (see, for example, Proposition 2.1 of \cite{un:BS}).  Fix piecewise Riemannian metrics $g_X$, $g_{P}$ on $X$ and $P$ respectively, where the metric on each simplex agrees with that of a Euclidean simplex whose edges all have length 1.  We can pull back the metric $f^*(g_{P})$ to $X$, where it is everywhere degenerate because the dimension of $P$ is strictly smaller than $m$.  Consider a new metric $g_s$ for $s>0$ defined pointwise by \[g_s=f^*(g_{P})+s^2g_X.\]
Since $f^*(g_{P})$ is everywhere degenerate, we clearly have \[\lim_{s\to 0^+}\vol(X,g_s)=0.\]
We will prove the theorem by showing that $\omega(X,g_s) = \entropy(X,g_s) \vol(X,g_s)^{1/m}$ goes to 0 as $s$ approaches 0.  It is not the case in general that $\entropy(X,g_s)$ stays bounded as $s \to 0^+$, nonetheless, we can show that $\omega(X,g_s)$ does limit to 0 as $s \to 0^+$.  

To calculate the volume entropy we will estimate the number of homotopy classes in $X$ of $g_s$--length at most $t$, as $t$ becomes large.  

Let us first estimate the number of homotopy classes in fibers. Since $g_s$ reduces to $s^2g_X$ along fibers, we can choose $s$ sufficiently small so that each fiber has diameter at most $\frac{1}{2}$.  Suppose $F_v=f^{-1}(v)$ is a fiber over some vertex $v\in P$. Fix a basepoint $x\in F_v$, and for $t \geq 0$ let $\calN_s(F_v,x;t)$ denote the number of homotopy classes of loops in $F_v$ based at $x$ whose $g_s$--length is at most $t$.  Then since $g_s$ scales the $g_X(=g_1)$--length of edges in $F_v$ by $s$, we have 
\begin{equation*}
\calN_s(F_v,x;t) = \calN_1\left(F_v,x;\frac{t}{s}\right).
\end{equation*}

As each fiber is subexponentially growing by asumption and since there are only finitely many vertices in $P$, for every $0< \lambda \leq 1$, there exists $T_\lambda$ such that for all $t \geq T_\lambda$ and vertex $v \in P$, we have
\begin{equation*}\label{eq:SubExp}
\calN_1\left(F_v,x;t \right)\leq \exp(\lambda T).
\end{equation*}
Define $\calN(t)=\max_{v \in P} \calN_1\left(F_v,x;t \right)$ and extend $\calN$ by linear interpolation to a continuous non-decreasing function $\calN \from [0,\infty) \to [0,\infty)$.  We define $C_\lambda = \calN(T_\lambda)$, so that for all vertices $v \in P$ and for all $t \geq 0$, \[\calN_1\left(F_v,x;t\right) \leq C_\lambda \exp(\lambda t).\]
 
For fixed $s$, we can let $\lambda$ depend on $s$.  Suppose that $\calN(t_0) > 1$ for some $t_0 > 0$ so that $\calN(t)$ satisfies the hypotheses of Lemma~\ref{lem:MovingZero}.  Let $\lambda_0 = \min\{ 1, \frac{1}{t_0}\log \calN(t_0) \}$ as defined in Lemma~\ref{lem:MovingZero}.  Assume that $s$ is small enough so that $\frac{1}{s} \geq T_{\lambda_0}$.  Since $T_\lambda$ is strictly decreasing on $(0,\lambda_0]$ and $\lim_{\lambda \to 0^+} T_\lambda = \infty$, there is a $\Lambda_{1/s} \in (0,\lambda_0]$ such that $T_{\Lambda_{1/s}} = \frac{1}{s}$.  If $\calN(t) \leq 1$ for all $t$, we take $\Lambda_{1/s} = 0$. 

An arbitrary loop $\gamma$ in $X$ can be represented as an edge path in the 1--skeleton.  Decompose such a path $\gamma$ as  \[\gamma = \varepsilon_1\sigma_1\cdots\varepsilon_k\sigma_k,\] where the $\varepsilon_i$ are long edges, and the $\sigma_i$ are edge paths consisting of edges in some fiber $F_{v_i}$. 

Connect the endpoints of each $\sigma_i$ to the basepoint $x_i \in F_{v_i}$ to form a loop $\overline{\sigma}_i$ at the expense adjoining to paths of length at most $\diam(F_{v_i},g_s) < \frac{1}{2}$. Hence if $\sigma_i$ has $g_s$--length $t_i$ then the length of $\overline{\sigma}_i$ is at most $t_i + 2\diam(F_{v_i},g_s) \leq  t_i+1$. Up to homotopy, the number of possible $\sigma_i$ of length at most $t_i$ is then bounded above by 
\begin{align*}
\calN_s(F_{v_i},x_i;t_i+2\diam(F_{v_i},g_s)) \leq \calN_1\left(F_{v_i},x_i;\frac{t_i+1}{s}\right) 
\leq C_{\Lambda_{1/s}} \exp\left(\Lambda_{1/s} \frac{t_i+1}{s}\right).
\end{align*}

Let $n_e$ be the total number of edges in $X$.  On the one hand, each long edge $\varepsilon_i$ has $g_s$--length at least 1, hence $k$ is less that than the length of $\gamma$.  On the other hand, we also have $\sum_{i=1}^k t_i$ is less than the length of $\gamma$.  Thus, we can bound the total number of possible paths of $g_s$--length at most some integer $t$ by 
\begin{align*}
n_e^t\prod_{i=1}^t C_{\Lambda_{1/s}} \exp\left(\Lambda_{1/s} \frac{t_i+1}{s}\right) & =
n_e^t C_{\Lambda_{1/s}}^t \exp\left( \frac{\Lambda_{1/s}}{s} \sum_{i=1}^t t_i\right) \exp\left(\frac{\Lambda_{1/s} t}{s}\right) \\
 &\leq n_e^t C_{\Lambda_{1/s}}^t \exp\left(\frac{\Lambda_{1/s} t}{s}\right)\exp\left( \frac{\Lambda_{1/s} t}{s}\right).
\end{align*}
 
Taking the logarithm, dividing by $t$ and letting $t \to \infty$ we obtain:
\begin{equation*}\label{eq:Logged}
\entropy(X,g_s) \leq \log(n_e) + \log(C_{\Lambda_{1/s}}) + \frac{2\Lambda_{1/s}}{s} \end{equation*}
Recall now that $C_{\Lambda_{1/s}} = \calN(T_{\Lambda_{1/s}}) = \calN\left(\frac{1}{s}\right)$.  Let $\epsilon = \frac{1}{2}(1/m - \nu)$ so that $\nu + 2\epsilon = 1/m$.  By Lemma~\ref{lem:MovingZero}\eqref{item:3} we have $\Lambda_{1/s} \leq s^{1-(\nu + \epsilon)}$ for sufficiently small $s$.  Thus, for such $s$, we have
\begin{equation*}\label{eq:EntEst}
\entropy(X,g_s) \leq \log(n_e)+\log \calN\left(\frac{1}{s}\right) + \frac{2}{s^{\nu + \epsilon}}.
\end{equation*}
Let $V=\vol(X,g_X)^{\frac{1}{m}}$ be the normalized volume of $X$ with the initial metric $g_X$. Therefore, the normalized volume of $g_s$ is $s^{1/m}V$. Multiplying the above by this we get
\begin{equation*}
\entropy(X,g_s)\vol(X,g_s)^{\frac{1}{m}} \leq s^{1/m}V\log(n_e) + s^{1/m}V\log \calN\left(\frac{1}{s}\right) + 2Vs^{\epsilon}.
\end{equation*}
The first and third terms on the right hand side clearly go to 0 as $s \to 0^+$.  The middle term goes to 0 as $s \to 0^+$ because $\calN$ is a subexponential function with subexponential growth rate less than $1/m$.  Thus, the left hand side must go to 0 as $s \to 0^+$, so the sequence of metrics $(X,g_s)$ shows that $\omega(X)$ must be 0.
\end{proof}

As a consequence, we get the following strengthening of Corollary~\ref{co:mve vanishes}.

\begin{proposition}\label{prop:mve vanishes}
Let $G$ be a group of finite type.  Then the following are equivalent:
\begin{enumerate}
\item $\omega(X) = 0$ for some $G$--complex $X$ with $\dim(X) = \gdim(G)$.
\item $\omega(X) = 0$ for every $G$--complex $X$.
\end{enumerate}
\end{proposition}

\begin{proof}
Suppose that $\omega(X) = 0$ where $\dim(X) = \gdim(G)$ and let $Y$ be a $G$--complex.  If $\dim(Y) = \dim(X)$, then $\omega(Y) = 0$ by Corollary~\ref{co:mve vanishes}.  Else we have that $\dim(Y) > \dim(X)$.  Let $f \from Y \to X$ be a homotopy equivalence.  As in Proposition~\ref{prop:essential he invariant}, we may assume that $f$ is simplicial.  As $f$ is a homotopy equivalence, every fiber $f^{-1}(p)$ has subexponential growth with subexponential growth rate 0.  Indeed, the inclusion induced image of $\pi_1(f^{-1}(p))$ in $\pi_1(Y)$ is trivial.  Therefore $Y$ satisifes the FCA and $\omega(Y) = 0$ by Theorem~\ref{thm:FCA no volume}.  This shows that (1) implies (2).

The other implication is obvious.
\end{proof}

\subsection{Fiber \texorpdfstring{$\pi_1$--growth}{pi\_1-growth} non-collapsing assumption}

Next, we discuss the non-collapsing assumption.

\begin{definition}[Babenko--Sabourau \cite{un:BS}]\label{def:FNCA} 
A simplicial complex $X$ of dimension $m$ satisfies the \emph{fiber $\pi_1$-growth non-collapsing assumption} (FNCA) if there exists a constant $\delta = \delta(X) > 0$ such that for every simplicial map $f \from X \to P$ to a finite simplicial complex $P$ of dimension at most $m-1$, there exists $p \in P$ and a connected component $F_0\subseteq f^{-1}(p)$ such that the inclusion induced image of $\pi_1(F_0)$ in $\pi_1(X)$ has uniform exponential growth at least $\delta$.
\end{definition}

Babenko--Sabourau prove that the FNCA is sufficient to ensure non-vanishing of the minimal volume entropy and moreover provide a positive lower bound in this case.

\begin{theorem}[{Babenko--Sabourau~\cite[Theorem~3.6]{un:BS}}]\label{thm:FNCA}If $X$ is a connected, finite simplicial complex with dimension $m$ satisfying the FNCA, then $\omega(X) > 0$.  More precisely, we have\[\omega(X)\geq \frac{\delta}{2 \cdot C_m} \] 
where $\delta = \delta(X)$ and $C_m>0$ is a constant depending only on the dimension $m$.
\end{theorem}

\begin{remark}
As we are primarily concerned with 2--dimensional simplicial complexes, we note that according to Papasoglu one may take $C_2= \const$~\cite{un:Papasoglu}.
\end{remark}

\subsection{FCA and FNCA are complementary for groups with Property \texorpdfstring{$U$}{U}}\label{sec:complementary}

As pointed out by Babenko--Sabourau, the defintions of FCA and FNCA are not complementary.  The subtlety lies in the subexponential growth rate in the definition of the FCA and the uniformity of the constant $\delta$ in the definition of the FNCA.  If we assume that the fundamental group of the complex has Property $U$, then this issue disappears.

\begin{lemma}\label{lem:complementary}
Let $G$ be a group of finite type and suppose that $G$ has Property $U$.  Then any $G$--complex either satisfies the FCA or the FCNA.
\end{lemma}

\begin{proof}
Let $G$ be as in the statement and let $X$ be a $G$--complex.

Suppose that $X$ does not satisfy the FCA.  Hence, given any simplicial map $f \from X \to P$ where $P$ is a simplicial complex with $\dim(P) < \dim(X)$, there is some point $p \in P$ and a component $F_0 \subseteq f^{-1}(p)$ such that the inclusion induced image of $\pi_1(F_0)$ in $\pi_1(X)$ has exponential growth.  As $G \cong \pi_1(X)$, we must have that the uniform growth rate of the inclusion induced image of $\pi_1(F_0)$ is at least $\udelta(G)$.  Thus we see that $X$ satisfies the FNCA for $\delta(X) = \udelta(G)$.  
\end{proof}

Combining Lemma~\ref{lem:complementary} with Theorem~\ref{thm:FCA no volume}, Proposition~\ref{prop:mve vanishes} and Theorem~\ref{thm:FNCA} we obtain the following dichotomy for any group of finite type.

\begin{proposition}\label{prop:complementary}
Let $G$ be a group of finite type with $m = \gdim(G)$ and suppose that $G$ has Property $U$.  Then either
\begin{enumerate}
\item $\omega(X) = 0$ for every $G$--complex, or
\item $\uomega(G) \geq \frac{\udelta(G)}{2 \cdot C_m}$.
\end{enumerate}
\end{proposition}

\begin{proof}
If some $G$--complex $X$ with $\dim(X) = \gdim(G)$ satisfies the FCA, then $\omega(X) = 0$ by Theorem~\ref{thm:FCA no volume}.  Thus by Proposition~\ref{prop:mve vanishes}, we get that $\omega(X) = 0$ for every $G$--complex $X$ and thus (1) holds.

Else, by Lemma~\ref{lem:complementary}, every $G$--complex $X$ with $\dim(X) = \gdim(G)$ satisfies the FNCA with $\delta(X) = \udelta(G)$.  Thus $\omega(X) \geq \frac{\udelta(G)}{2 \cdot C_m}$ by Theorem~\ref{thm:FNCA}.  As $X$ is an arbitrary $G$--complex with $\dim(X) = \gdim(G)$, we see that (2) holds.
\end{proof}

%%%%%%%%%%%%%%%%%%%%%%%%%%%%%%%%%%%%%%%%%%%%%%%%%%%%%%%%%%%%%%%%%%%%%%%%%%%%%

\section{Vanishing Criterion when \texorpdfstring{$\gdim(G) = 2$}{gd(G) = 2}}\label{sec:vanishing criterion}

In this section we prove the first main result of this paper.  Theorem~\ref{thm:zero iff graph of groups} provides a characterization for when $\omega(X) = 0$ for every $G$--complex $X$, if $\gdim(G) = 2$, $\udelta(G) > 0$, and the subexponentially growing subgroups of $G$ belong to the collection $\{\{1\}, \ZZ, \ZZ^2, BS(1,-1) \}$.  (In particular, such a $G$ has Property $U$.) Specifically, the minimal volume entropy vanishes for each $G$--complex precisely when $G$ is the fundamental group of a graph of groups where the edge groups belong to the collection $\{\{1\}, \ZZ \}$ and the vertex groups belong to the collection $\{\ZZ, \ZZ^2, BS(1,-1)\}$.  Before we prove this theorem in Section~\ref{subsec:zero iff graph of groups}, we recall the definition of a graph of groups and set some notation in Section~\ref{subsec:graph of groups}.  Then in Section~\ref{subsec:fca -> graph}, we show that if a $G$--complex of minimal dimension satisfies the FCA and $\gdim(G) = 2$, then $G$ is the fundamental group of a graph of groups where the vertex groups and edge groups are subexponentially growing.  We complete the proof of Theorem~\ref{thm:zero iff graph of groups} in Section~\ref{subsec:zero iff graph of groups}.

\subsection{Graphs of Groups}\label{subsec:graph of groups}

General references for the material in this section are the works of Bass~\cite{ar:Bass93}, Scott--Wall~\cite{col:SW79}, and Serre~\cite{bk:Serre03}.  

A \emph{graph of groups} consists of the following data.
\begin{enumerate}
\item A finite connected graph $Y$ with vertex set $VY$ and edge set $EY$. By $o(e)$ and $\tau(e)$ we denote the originating and terminal vertices of an edge $e$ respectively, and $\bar{e}$ denotes the edge with opposite orientation.  We have $\bar{\bar{e}} = e$ and $o(\bar{e}) = \tau(e)$.
\item For each vertex $v \in VY$, there is an associated group $G_v$.
\item For each edge $e \in EY$, there is an associated group $G_e$.  We have $G_{\bar{e}} = G_e$.
\item For each edge $e \in EY$, there is an injective homomorphism $h_e \from G_e \to G_{o(e)}$.
\end{enumerate}
We will denote a graph of groups by $\calG = (Y,\{G_v\},\{G_e\},\{h_e\})$.  

Associated to a graph of groups $\calG = (Y,\{G_v\},\{G_e\},\{h_e\})$ is the \emph{fundamental group of the graph of groups}, denoted $\pi_1(\calG)$.  Briefly, it is constructed by repeatedly taking amalgamated free products and HNN-extensions using the data in $\calG$.  In more detail, as a generating set of $\pi_1(\calG)$ we take the set 
\begin{equation*}
\{G_v \mid v \in VY \} \cup \{x_e \mid e \in EY \}.
\end{equation*}

All of the relations in the vertex groups hold plus some more that use the data in $\calG$.  To write down these additional relations for $\pi_1(\calG)$, we need to fix a maximal tree $T \subseteq Y$.  Using the tree $T$, the additional relations for $\pi_1(\calG)$ are as follows:
\begin{align*}
x_e h_e(a) & = h_{\bar{e}}(a)x_{\bar{e}}, \text{ for each edge $e \in EY$ and element $a \in G_e = G_{\bar{e}}$} \\
x_e x_{\bar{e}} & = 1, \text{ for each edge $e \in EY$} \\
x_e & = 1, \text{ for each edge $e \in ET$} 
\end{align*}
The isomorphism type of $\pi_1(\calG)$ does not depend on the choice of maximal tree $T \subseteq Y$. 

Consider two graphs of groups $\calG = (Y,\{G_v\},\{G_e\},\{h_e\})$ and $\calG' = (Y',\{G'_v\},\{G'_e\},\{h'_e\})$ where
\begin{enumerate}
\item $Y'$ is a subgraph of $Y$,
\item $G'_v = G_v$ for each vertex $v \in VY'$,
\item $G'_e = G_e$ for each edge $e \in EY'$, and
\item $h'_e = h_e$ for each edge $e \in EY'$.
\end{enumerate}
Then $\pi_1(\calG')$ is isomorphic to a subgroup of $\pi_1(\calG)$.

Given a graph of groups $\calG = (Y,\{G_v\},\{G_e\},\{h_e\})$, there is an associated \emph{graph of spaces} $\calX = (Y,\{X_v\},\{X_e\},\{f_e\})$, which is well-defined up to homotopy.  For each vertex $v \in VY$, we set $X_v = K(G_v,1)$.  Likewise, for each edge we set $X_e = K(G_e,1)$ and further we fix a map $f_e \from X_e \to X_v$ so that $(f_e)_* = h_e$.   There is an associated space $\abs{\calX}$, called the \emph{realization of the graph of spaces}, obtained by gluing the spaces together using the graph $Y$.  Specifically, we define
\begin{equation*}
\abs{\calX} =  \raisebox{0.2cm}{$\displaystyle \bigcup_{v \in VY} X_v \cup \bigcup_{e \in EY} X_e \times [0,1]$} \Bigg/
      \raisebox{-0.2cm}{\begin{varwidth}{10cm}$(x,0) \in X_e \times [0,1] \sim f_e(x) \in X_{o(e)}$ and \\ 
      $(x,1) \in X_e \times [0,1] \sim (x,1) \in X_{\bar{e}} \times [0,1]$\end{varwidth}}
\end{equation*} 
We have that $\pi_1(\abs{\calX}) \cong \pi_1(\calG)$.

Let $\calG = (Y,\{G_v\},\{G_e\},\{h_e\})$ be a graph of groups.  Suppose that $e_0$ is an edge in $Y$ so that $o(e_0) \neq \tau(e_0)$, i.e., $e_0$ is not a loop.  If the inclusion map $h_{e_0} \from G_{e_0} \to G_{o(e_0)}$ is an isomorphism, then we say that the edge $e_0$ is \emph{collapsible}.  In this case, we may collapse $e_0$ and obtain a new graph of groups $\calG' = (Y',\{G'_v\},\{G'_e\},\{h'_e\})$.  The underlying graph $Y'$ is obtained by removing the edge $e_0$ from $Y$ and identifying the vertices $o(e_0)$ and $\tau(e_0)$; we denote this image of these vertices by $v'$ and define $G'_{v'} = G_{\tau(e)}$.  All other vertices and edges of $Y'$ correspond to a vertex or edge of $Y$ and we define the vertex group $G'_v$ or edge group $G'_e$ accordingly.  As the map $h_{e_0} \from G_{e_0} \to G_{o(e_0)}$ is an isomorphism, we can consider $G_{o(e_0)}$ as a subgroup of $G_{\tau(e_0)}$ via $h_{\bar{e}_0}h_{e_0}^{-1}$.  Thus for an edge $e$ in $Y$ where $o(e) = o(e_0)$, the injective homomorphism $h_e \from G_e \to G_{o(e)}$ naturally defines an injective homomorphism $h'_e \from G'_{e} \to G'_{o(e)}$.  For all other edges, we have that $h'_e = h_e$.  We say that $\calG'$ is obtained from $\calG$ by \emph{collapsing the edge $e$}.  This does not change the fundamental group, i.e., $\pi_1(\calG) \cong \pi_1(\calG')$.  This follows because of the isomorphism $A \ast_C C \cong A$.  

If no edge of $Y$ is collapsible, we say that $\calG$ is \emph{reduced}.  If $\calG$ is not reduced, we may repeatedly collapse edges to obtain a reduced graph of groups decomposition whose fundamental group is $\pi_1(\calG)$.  

There is a correspondence between decompositions of $G$ as a graphs of groups, i.e., isomorphisms $G \cong \pi_1(\calG)$, and actions of $G$ on simplicial trees $\tY$.  In this correspondence, the vertex groups and edge groups of $\calG$ correspond to the conjugacy classes of the vertex stabilizers and edge stabilizers respectively for the action of $G$.  The underlying graph of $\calG$ is $G \backslash \tY$.

\subsection{FCA induces a graph of groups decomposition}\label{subsec:fca -> graph}

We will now show how the FCA induces a graph of groups decomposition when $\gdim(G) = 2$.  

\begin{proposition}\label{prop:fca -> graph}
Suppose that $G$ is a group with $\gdim(G) = 2$.  If $X$ is a $G$--complex with $\dim(X) = 2$ and $X$ satisfies the FCA, then $G$ is isomorphic to the fundamental group of a graph of groups where the vertex groups and the edge groups are subexponentially growing.
\end{proposition}

\begin{proof}
Let $G$ be as in the statement and let $X$ be a $G$--complex that satisfies the FCA.  Hence, there is a graph $\Gamma$ and a simplicial map $f \from X \to \Gamma$ with connected fibers such that for each $x \in \Gamma$, the image of $\pi_1(F_x)$ in $\pi_1(X)$ has subexponential growth where $F_x = f^{-1}(x)$.  

The function $f \from X \to \Gamma$ induces a graph of groups decomposition of $G$ as in the statement of the proposition as we now recall.  For further details, we refer the reader to the work of Dunwoody~\cite{ar:Dunwoody85}.  For each edge $e$ of $\Gamma$, fix a point $x_e$ in the interior of the edge.  Let $\tX$ be the universal cover of $X$ and let $p \from \tX \to X$ be the covering map.  We consider the lift of the fibers $F_{x_e}$ to $\tX$ and define \[\calE = \cup \{\pi_0(p^{-1}(F_{x_e})) \mid e \mbox{ is an edge of } \Gamma \} \mbox{ and }\calV = \pi_0(\tX - \cup \{ \varepsilon \mid \varepsilon \in \calE \}).\]    
For each $\varepsilon \in \calE$, there are exactly two components in $\calV$ whose closures in $\tX$ contain $\varepsilon$.  Therefore, in the obvious way, we get a graph $T$ with vertices $\calV$ and edges $\calE$.  The graph $T$ is clearly connected and as each $\varepsilon \in \calE$ separates $\tX$, we see that $T$ is a tree.

The stabilizer of a component in $\calE$ is a conjugate of the image of $\pi_1(F_{x_e})$ in $\pi_1(X)$ for some edge $e$ of $\Gamma$.  Given a component $c \in \calV$, the subset $p(c)$ deformation retracts onto $F_v$ for some vertex $v$ of $\Gamma$.  Hence the stabilizer of a component in $\calV$ is a conjugate of the image of $\pi_1(F_v)$ in $\pi_1(X)$ for some vertex $v$ of $\Gamma$.

Therefore $G$ acts on a tree where the stabilizer of any point is subexponentially growing.  As stated in Section~\ref{subsec:graph of groups}, by Bass--Serre theory this implies that $G$ is isomorphic to the fundamental group of a graph of groups where the vertex groups and edge groups are subexponentially growing.
\end{proof}

\subsection{Proof of Theorem~\ref{thm:zero iff graph of groups}}\label{subsec:zero iff graph of groups}

Before can prove Theorem~\ref{thm:zero iff graph of groups}, we need a lemma that shows that certain subgroups are prohibited in groups with $\gdim(G) = 2$.

\begin{lemma}\label{lem:some subgroups}
Suppose that $H_1$, $H_2$ and $K$ belong to the collection $\{\ZZ^2,BS(1,-1)\}$.  The geometric dimension of an amalgamated free product $H_1 *_K H_2$ is equal to $3$ if both inclusions are proper.  The geometric dimension of an HNN-extension $H_1*_K$ is equal to $3$.
\end{lemma}

\begin{proof}
Let $H_1$, $H_2$ and $K$ be as in the statement.  For each of $H_1 \ast_K H_2$ and $H_1 \ast_K$, the respective graphs of spaces using $S^1 \times S^1$ for each $\ZZ^2$ and the Klein bottle for each $BS(1,-1)$ are aspherical and have dimension equal to $3$.  Thus $\gdim(H_1 \ast_K H_2) \leq 3$ and $\gdim(H_1 \ast_K) \leq 3$.  

The cohomological dimension of $H_1 \ast_K H_2$ is equal to $3$ if both inclusions are proper.  Likewise, the cohomological dimension $H_1 \ast_K$ is equal to $3$.  See the work of Bieri~\cite[Corollaries~6.5~and~6.7]{bk:Bieri76} for complete details.  As the geometric dimension is bounded from below by the cohomological dimension, the result follows. 
\end{proof}

\begin{proof}[Proof of Theorem~\ref{thm:zero iff graph of groups}]
Suppose that $G$ is a group with $\gdim(G) = 2$, $\udelta(G) > 0$, and that every subexponentially growing subgroup belongs to the collection $\{\{1\}, \ZZ, \ZZ^2, BS(1,-1)\}$.

First, we assume that $\omega(X) = 0$ for every $G$--complex $X$.  As $\udelta(G) > 0$, by Theorem~\ref{thm:FNCA} and Lemma~\ref{lem:complementary} there is some $G$--complex $X$ that satisfies the FCA.  Therefore by Proposition~\ref{prop:fca -> graph} we have that $G$ is isomorphic to the fundamental group of a graph of groups $\calG = (Y,\{G_v\},\{G_e\},\{h_e\})$ with subexponentially growing vertex groups and edge groups.  By collapsing any collapsible edges, we may assume that $\calG$ is reduced.      

If $G_v$ is trivial for some vertex $v \in VY$, then every edge incident on $v$ is a loop as $\calG$ is reduced.  This implies that $G \cong \pi_1(\calG)$ is a free group.  This is contrary to the assumption that $\gdim(G) = 2$.  Hence the vertex groups of $\calG$ belong to the collection $\{\ZZ,\ZZ^2,BS(1,-1)\}$.

By Lemma~\ref{lem:some subgroups}, the groups $\ZZ^2$ and $BS(1,-1)$ cannot appear as edge groups since $\calG$ is reduced and $\gdim(G) = 2$.  Thus, the edge groups of $\calG$ belong to the collection $\{ \{1\}, \ZZ \}$.  

Next, we assume that $G$ is isomorphic to the fundamental group of a graph of groups $\calG = (Y,\{G_v\},\{G_e\},\{h_e\})$ where vertex groups belong to the collection $\{\ZZ, \ZZ^2, BS(1,-1) \}$ and the edge groups belong to the collection $\{ \{ 1 \},\ZZ \}$.  Let $\calX = (Y,\{X_v\},\{X_e\},\{f_e\})$ be the corresponding graph of spaces built using a point, $S^1$, $S^1 \times S^1$ and the Klein bottle respectively for each $\{1\}$, $\ZZ$, $\ZZ^2$ and $BS(1,-1)$ respectively.  Then $\abs{\calX}$ is a $G$--complex with $\dim(\abs{\calX}) = 2 = \gdim(G)$ and there is a map $p \from \abs{\calX} \to Y$ where each of the fibers either a point, $S^1$, $S^1 \times S^1$ or $BS(1,-1)$.  This shows that $\abs{\calX}$ satisfies the FCA and hence $\omega(\abs{\calX}) = 0$ by Theorem~\ref{thm:FCA no volume}.  By Proposition~\ref{prop:mve vanishes}, we conclude that $\omega(X) = 0$ for every $G$--complex $X$.
\end{proof}

\section{Free-by-cyclic groups}\label{sec:FbyZ}

In this section we examine the minimal volume entropy of free-by-cyclic groups and prove Theorem~\ref{thm:FbyZ}.  To prove this theorem we must show that the following three statements are equivalent for a free-by-cyclic group $G_\phi$.  
\begin{enumerate}
\item $\omega(X)=0$ for every $G_\phi$--complex $X$.
\item $G_\phi$ is virtually tubular.
\item Some power of $\phi$ is geometrically linear unipotent power.
\end{enumerate}
First we prove that (1) implies (3) in Proposition~\ref{prop:omega = 0 -> glu}.  This takes place in Section~\ref{subsec:primitive free and glu} after we formally define a geometrically linear unipotent outer automorphism.  Following this, in Section~\ref{subsec:FbyZ proof} we complete the proof of Theorem~\ref{thm:FbyZ} by showing that (3) implies (2) and observing that (2) implies (1) by Theorem~\ref{thm:zero iff graph of groups} and Proposition~\ref{prop:finite index}.

Before we begin the proof of Theorem~\ref{thm:FbyZ}, in Section~\ref{subsec:growth of subgroups} we  classify subexponentially growing subgroups of free-by-cyclic groups and show that free-by-cyclic groups have uniform uniform exponential growth (with a uniform constant).

\subsection{Growth of subgroups of free-by-cyclic groups}\label{subsec:growth of subgroups}
Let $\phi$ be an outer automorphism of a finitely generated free group $F_n$ and let $G_\phi$ be the corresponding free-by-cyclic group. 

\begin{lemma}\label{lem:FbyZ subexponential subgroups}
Any nontrivial finitely generated subgroup of $G_\phi$ with subexponential growth is isomorphic to $\ZZ$, $\ZZ^2$ or $BS(1,-1)$.  
\end{lemma}

\begin{proof}
Write $G_\phi=F_n\rtimes_\Phi \ZZ$ where $\Phi \in \Aut(F_n)$ represents the outer automorphism $\phi$ and let $\pi$ be the projection onto the cyclic factor.  

Let $H$ be a nontrivial finitely generated subgroup of $G_\phi$ that has subexponential growth.  As $H$ is nontrivial and has subexponential growth, $H \cap F_n$ is either trivial or isomorphic to $\ZZ$.  If $H \cap F_n$ is trivial, then $\pi$ maps $H$ injectively to $\ZZ$ hence $H \cong \ZZ$.  Otherwise, we have that $H\cap F_n = \I{a}$ for some nontrivial $a \in F_n$. If $\pi(H)$ is trivial then $H= \I{a} \cong \ZZ$.  Otherwise, let $h$ be an element of $H$ that generates $\pi(H) \cong \ZZ$. Then $b=hah^{-1} \in H \cap F_n$ and so $b=a$ or $b=a^{-1}$.  As $h$ generates $\pi(H)$, this implies that $H \cong \ZZ^2$ in the case $b=a$ and that $H \cong BS(1,-1)$ in the case $b=a^{-1}$. 
\end{proof}

\begin{lemma}\label{lem:uniform uniform FbyZ}
Suppose $H$ is a finitely generated subgroup of $G_\phi$ that is exponentially growing.  Then $\delta(H) \geq \frac{1}{6}\log 3$.  In particular, $\udelta(G_\phi) \geq \frac{1}{6} \log 3$.
\end{lemma}

\begin{proof}
Write $G_\phi=F_n\rtimes_\Phi \ZZ$ where $\Phi \in \Aut(F_n)$ represents the outer automorphism $\phi$.

Suppose $H$ is a finitely generated subgroup of $G_\phi$ with exponential growth and let $S$ be  a finite generating set for $H$.  The commutator subgroup $[H,H]$ is contained in $F_n$ and is normally generated by $C=\{[s_i,s_j]\mid s_i,~s_j\in S\}$.  As $H$ has exponential growth we must have that $[H,H]$ is nontrivial.

Suppose that $[H,H]=\I{c}$ is infinite cyclic, generated by some element $c \in F_n$.  If so, there is a short exact sequence 
\[ 1 \to \I{c} \to H \to A \to 1 \] where $A$ is a finitely generated abelian group. In particular, an index two subgroup of $H$ is nilpotent.  This implies that $H$ has subexponential growth, contrary to our hypothesis. 

Therefore, $[H,H]$ is not cyclic and hence free and nonabelian.  Thus we can find two elements of $[H,H]$ of length at most $6$ (we may have to conjugate some $[s_i,s_j]$ by an element of $S$) that generate a group isomorphic to $F_2$, so $\delta(H,S) \geq \frac{1}{6}\log 3$.

As $S$ was an arbitrary generating set for $H$, we have $\delta(H) \geq \frac{1}{6}\log 3$.
\end{proof}

Lemmas~\ref{lem:FbyZ subexponential subgroups} and \ref{lem:uniform uniform FbyZ} show that free-by-cyclic groups have Property $U$.

\subsection{Geometrically linear unipotent outer automorphisms}\label{subsec:primitive free and glu}

In this section we will define a geometrically linear unipotent outer automorphism of a free group and prove that if $\omega(X) = 0$ for every $G_\phi$--complex $X$, then some power of $\phi$ is geometrically linear unipotent.

\begin{definition}\label{def:primitive free}
A \emph{primitive free splitting of $F_n$} is a graph of groups decomposition of $F_n$ where each vertex group is $\ZZ$ and each edge group is trivial. 
\end{definition}

Given a primitive free splitting of $F_n$, $\calF = (Y',\{F_v\},\{F_e\},\{i_e\})$, a model for the corresponding graph of spaces is obtained by attaching a loop edge $\alpha_v$ to each vertex $v \in VY'$.  We will denote this graph by $K(\calF)$.  By choosing a maximal tree $T\subseteq Y'$ (equivalently a maximal tree in $K(\calF)$, and fixing generators $a_v \in F_v$, a subset $E_+(Y'-T) \subset EY' - ET$ that contains one edge from each pair $\{e,\bar {e}\} \subseteq EY' - ET$ and a vertex $v_0\in VY'$, the set \[\{a_v \mid v \in VY' \} \cup \{ x_e \mid e \in E_+(Y' - T) \} \] is a basis for $F_n$ via the isomorphism $F_n \cong \pi_1(K(\calF),v_0)$.
 
\begin{definition}\label{def:tubular aut}
An outer automorphism $\phi \in \Out(F_n)$ is \emph{geometrically linear unipotent} (GLU) if there is a representative $\Phi \in \Aut(F_n)$, a primitive free splitting of $F_n$, $\calF = (Y',\{F_v\},\{F_e\},\{i_e\})$ of $F_n$, a maximal tree $T \subseteq Y'$ and vertex $v_0 \in VY'$ such that the following holds. 
\begin{enumerate}
\item For every $e \in ET$ where $o(e)$ lies between $v_0$ and $\tau(e)$, there is an integer $p_e$.
\item For each $v \in VY'$, $\Phi(a_v) = w_v a_v w_v^{-1}$ where $(e_1,e_2,\ldots,e_m)$ is the minimal length edge path in $T$ from $v_0$ to $v$ and \[ w_v = a_{o(e_1)}^{p_{e_1}}a_{o(e_2)}^{p_{e_2}} \cdots a_{o(e_m)}^{p_{e_m}}. \]
\item For each $e \in E_+(Y'-T)$, $\Phi(x_e) = w_{o(e)}a_{o(e)}^{q_e}x_ea_{\tau(e)}^{r_e}w_{\tau(e)}^{-1}$ for some $q_e,r_e \in \ZZ$.
\end{enumerate}
\end{definition}

\begin{example}\label{ex:Primitive}
Consider the primitive free splitting of $F_3 = \I{a,b,c}$, where the underlying graph has two vertices $v_1,v_2$ and two (geometric) edges $e_1,e_2$ where $o(e_1) = \tau(e_2) = v_1$ and $\tau(e_1) = o(e_2) = v_2$, and where the vertex groups are $F_{v_1} = \I{a}$ and $F_{v_2} = \I{b}$.  Let $T$ be the single edge $e_1$. See Figure \ref{fig:Primitive} below.

Fix integers $p$, $q$, and $r$.  The automorphism
\begin{equation*}
\Phi(a) = a; \quad \Phi(b) = a^p ba^{-p}; \quad \Phi(c) = a^p b^q c a^r
\end{equation*}
represents a GLU outer automorphism.
\end{example}

\begin{figure}[h]
    \centering
    \begin{tikzpicture}[every loop/.style={}]
        \draw[red,directed] (-1,0) arc (150:30:1);
        \draw[rdirected] (-1,0) arc (210:330:1);
        
        \filldraw (-1,0) circle (2pt);
        \filldraw (.75,0) circle (2pt);
        
        \node[left] at (-1,0) {\scriptsize $\I{a}$};
        \node[right] at (.75,0) {\scriptsize $\I{b}$};
        
        \draw[red,directed] (4,0) arc (150:30:1);
        \draw[rdirected] (4,0) arc (210:330:1);
        
        \draw[directed] (4,0)  to[in=270,out=220](2.75,0) to[in=140, out=90] (4,0);
        \draw[rdirected] (5.75,0)  to[in=270,out=320](7,0) to[in=40, out=90] (5.75,0);
        
        \fill (4,0) circle[radius=2pt];
        \fill (5.75,0) circle[radius=2pt];
        
        \node[left] at (2.75,0) {\scriptsize $a$};
        \node[right] at (7,0) {\scriptsize $b$};

    \end{tikzpicture}
    \caption{The primitive free splitting of Example \ref{ex:Primitive} is shown on the left, and its associated geometric realization on the right. The tree $T$ is colored in red.  }\label{fig:Primitive} 
\end{figure}
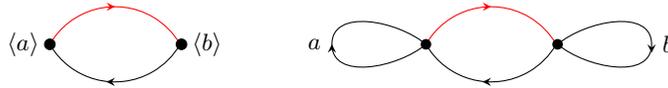

\begin{remark} \label{rmk:tubular form}
From the definition, we see that GLU outer automorphisms are linearly growing. Moreover, there exists a topological representative $f \from  K(\calF) \to K(\calF)$ for $\phi$ defined by the following.
\begin{enumerate}
\item For every $v\in VY'$, $f(\alpha_v)=\alpha_v$.
\item For every $e\in ET$ where $o(e)$ lies between $v_0$ and $\tau(e)$,  $f(e)=\alpha_{o(e)}^{p_e} e$.
\item For every $e\in E_+(Y' - T)$, $f(e)=\alpha_{o(e)}^{q_e} e \alpha_{\tau(e)}^{r_e}$.
\end{enumerate}
Conversely, any outer automorphism that has such a geometric representative is GLU.
\end{remark}

Thus, each GLU automorphism has a geometric representative of a highly restrictive form.  In the action of a GLU automorphism on the abelianization of $F_n$, it will be represented by a linearly growing unipotent matrix. Of course, not every automorphism with this property will be GLU.  The terminology \emph{geometrically} linear unipotent is meant to indicate that even on the level of homotopy the automorphism resembles a linear unipotent automorphism.

\begin{proposition}\label{prop:omega = 0 -> glu}
If $\omega(X)=0$ for every $G_\phi$--complex $X$, then some power of $\phi$ is geometrically linear unipotent. 
\end{proposition}

\begin{proof}
Write $G_\phi=F_n\rtimes_\Phi \ZZ = \I{F_n,t \mid tat^{-1} = \Phi(a), \, \forall a \in F_n}$ where $\Phi \in \Aut(F_n)$ represents the outer automorphism $\phi$.  Suppose that $\omega(X) = 0$ for every $G_\phi$--complex $X$.  We will show that the automorphism obtained by replacing $\Phi$ with a power and composing the result by an inner automorphism satisfies the conditions in Definition~\ref{def:tubular aut}.  On the level of the presentation, this is accomplished by replacing $t$ with $at^k$ for some $k \in \ZZ$ and $a \in F_n$.

As $\gdim(G_\phi) = 2$, Lemmas~\ref{lem:FbyZ subexponential subgroups} and~\ref{lem:uniform uniform FbyZ} show that we may apply Theorem~\ref{thm:zero iff graph of groups} to $G_\phi$.  Since $G_\phi$ is 1--ended, Theorem~\ref{thm:zero iff graph of groups} yields a graph of groups decomposition $\calG = (Y,\{G_v\},\{G_e\},\{h_e\})$ of $G_\phi$ where all vertex groups belong to the collection $\{\ZZ, \ZZ^2, BS(1,-1)\}$ and every edge group is equal to $\ZZ$.  We may assume that $\calG$ is reduced. 

Let $\tY$ be the tree corresponding to the decomposition $G_\phi \cong \pi_1(\calG)$, so that $G_\phi\backslash \tY=Y$.  The normal subgroup $F_n \subset G_\phi$ acts on $\tY$ and the quotient $Y'=F_n\backslash \tY$ yields a graph of groups decomposition $\calF=(Y',\{F_v\},\{F_e\},\{i_e\})$ of $F_n$.  Let $\pi\colon Y'\rightarrow Y$ be the quotient map induced the action of $\I{t}$.  Following an observation of Brinkmann \cite{ar:Brinkmann00}, we claim:

\begin{claim}
The graph $Y'$ is finite, each vertex group $F_v$ belongs to the collection $\{\{1\},\ZZ\}$ and each edge group $F_e$ is trivial.
\end{claim}

Indeed, the quotient map $\pi$ induces an injection on edge and vertex groups.  In particular, as $\calF$ is a decomposition of $F_n$ and all vertex groups in $\calG$ are either $\ZZ,~\ZZ^2$ or $BS(1,-1)$, we must have that each vertex group of $\calF$ is either $\ZZ$ or trivial.  

For any edge $e$ of $Y$, the map $\rho_e \from Y \to Y_e$ that collapses the components of the complement of $e$ induces a graph of groups decomposition $\calG_e$ for $G_\phi$ with a single edge where the edge group is $\ZZ$, i.e., a splitting over $\ZZ$.  That is, there is a tree $\tY_e$ and $G_\phi$--equivariant map $\widetilde{\rho}_e \from  \tY \to \tY_e$ with connected fibers inducing the map $\rho_e$ on quotients where $G_\phi \backslash \tY_e$ consists of a single edge.  In particular, the edge group of $e$ is the same in either graph of groups.  Brinkmann showed that for any splitting of a free-by-cyclic group over $\ZZ$, the induced graph of groups decomposition on $F_n$, $\calF_e = (Y'_e,\{F'_v \},\{F'_e\},\{i'_e\})$ has $Y'_e$ a finite graph and $F'_e = \{1\}$ for each edge~\cite[Section~1]{ar:Brinkmann00}.   
Therefore for each $e$ we have a pullback square
\[\xymatrix{Y'\ar[r]^\pi\ar[d]_{\rho_e'} & Y\ar[d]^{\rho_e}\\ Y_e'\ar[r]_{\pi_e}&Y_e}\]
The map $\rho'_e \from Y' \to Y'_e$ collapses the components of the complement of $\pi^{-1}(e)$.  As above, the edge groups for $\calF_e$ are the same as the corresponding edge groups in $\calF$.  Hence the edge group for an edge in $\pi^{-1}(e)$ is trivial.  As $e$ was arbitrary, this shows that the edges groups in $\calF$ are all trivial.  Further, as $Y'_e$ is a finite graph, we see that $\pi^{-1}(e)$ consists of finitely many edge for any edge $e$ of $Y$.  This shows that $Y'$ is a finite graph.  This proves the claim.

\medskip 

Since $Y'$ is finite, the stable letter $t$ acts on $Y'$ by a finite order automorphism.  Thus some power of $t$, $t^k$, acts as the identity on $Y'$ and also on each vertex group $F_v$ since each such group has at most two automorphisms.  We replace $G_\phi$ with the finite index subgroup $G_{\phi^k}$.  For the action of $G_\phi$ on $\tY$ we now have that $G_\phi \backslash \tY = Y' = F_n \backslash \tY$.  We continue to denote the graph of group decomposition of $G_\phi$ by $\calG$.  

The graph of groups decomposition $\calF$ may not be reduced.  As $t$ acts trivially on $Y'$, if an edge $e$ is collapsible for $\calF$, it is also collapsible for $\calG$.  This follows as the $G_\phi$--stabilizer of an edge in $\tY$ is generated by $at$ for some $a \in F_n$ since $t$ acts as the identity on $Y'$ and such an element does not have a proper root.  Therefore, we may collapse edges in $Y'$ so that $\calF$ is reduced.  We will continue to denote by $Y'$ the underlying graph.

% As $\calG$ is reduced, so is $\calF$.  Indeed, suppose that $e$ is an edge of $Y$ that is collapsible for $\calF$.  Thus both groups $F_v$ and $F_e$ are trivial.  The group $G_e$ is equal to $\ZZ$ and generated by some element $g = at^m \in G_\phi$ for some $m \in \ZZ$ and $a \in F_n$.  No proper root of $g$ can lie in $G_v$ as $t$ acts as the identity on $Y$.  Thus, if $G_v$ does not equal $\I{g}$, then there exists some element $h = bt^k$ for some $k \in \ZZ$ and $b \in F_n$ such that $\I{g,h}$ is either $\ZZ^2$ or $BS(1,-1)$.  In this case, the element $g^kh^{-m}$ is nontrivial and lies in $F_v$.  This contradiction implies that $G_v$ does equal $G_e$ and so the edge $e$ is collapsible in $\calG$ too.     

If $Y'$ has a single vertex $v$ and $F_v$ is trivial, then we collapse one of the incident loops and change the vertex group to $\ZZ$.  Hence $F_v = \ZZ$ for all $v\in Y'$ and $F_e=\{1\}$ for all $e \in EY'$.  In other words, $\calF$ is a primitive free splitting. 

Let $T\subseteq Y'$ be a maximal tree.  Choose a basepoint $v_0\in T$ and fix a subset $E_+(Y' - T) \subset EY' - ET$ that contains one edge from each pair $\{e,\bar{e}\}$. A lift $\tT \subseteq\tY$ of $T$, determines for each $v\in \tT$ an element $a_v$ that generates $F_v$.  For each $e\in E_+(Y' - T)$ we obtain a hyperbolic element $x_e \in F_n$ that identifies an edge $\tilde{e}_1$---which is a lift of $e$---at the lift of $o(e)$ in $\tT$, with an edge $\tilde{e}_2$---which is also a lift of $e$---at the lift of $\tau(e)$ in $\tT$. See Figure \ref{fig:translate}.

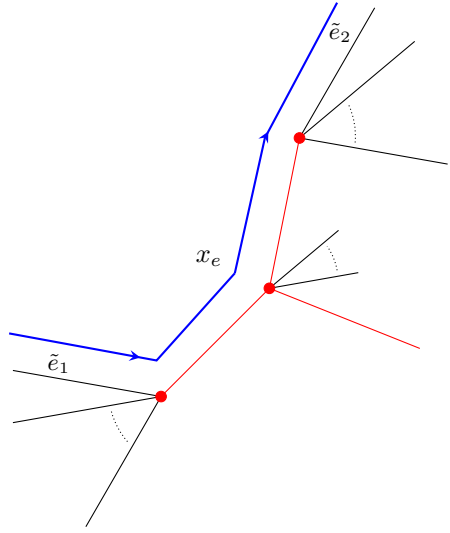
\begin{figure}[ht]
    \centering
    \begin{tikzpicture}[scale=2]
        \draw[red] (-.72,-.72) -- (0,0);
        \draw[red] (0,0)--(1,-.4);
        \draw[red] (0,0)--(.2,1);
       
        \draw[shift={(-.72cm,-.72cm)}] (0,0) -- (170:1);
        \draw[shift={(-.72cm,-.72cm)}] (0,0) -- (240:1);
        \draw[shift={(-.72cm,-.72cm)}] (0,0) -- (190:1);
        \draw[densely dotted,shift={(-.33,-.1)}] (-.72,-.72) arc (195:222:.5);
        
        \draw[shift={(.2,1)}] (0,0) -- (60:1);
        \draw[shift={(.2,1)}] (0,0) -- (40:1);
        \draw[shift={(.2,1)}] (0,0) -- (-10:1);
        \draw[densely dotted,shift={(.37,-.03)}] (.2,1) arc (-5:25:.5);
        
        \draw (0,0) -- (40:.6);
        \draw (0,0) -- (10:.6);
        \draw[densely dotted,shift={(.45,.12)}] (0,0) arc (10:35:.4);
        \draw[thick, blue, directed, shift={(-.03,.1)}] (-1.7,-.4)-- (-.72,-.58) -- (-.2,0);
        \draw[thick, blue, directed, shift={(-.03,.1)}] (-.2,0) -- (0,.9)--(.48,1.8);

        % \draw[blue, shift={(-.1,.1)}] (-.72,-.72) -- (0,0);
        % \draw[blue, shift={(-.1,.1)}] (0,0)--(.2,1);
        % \draw[blue, shift={(-.82,-.62)}] (0,0) -- (170:1);
        % \draw[blue, shift={(.1,1.1)}] (0,0) -- (60:1);

        \filldraw[red] (.2,1) circle (1pt) ;
        \filldraw[red] (0,0) circle (1pt);
        \filldraw[red] (-.72,-.72) circle (1pt);
        
        \node at (.47,1.7) {\scriptsize $\tilde{e}_2$};
        \node at (-1.4,-.5) {\scriptsize $\tilde{e}_1$};
        \node at (-.4,.2) {\footnotesize $x_e$};
\end{tikzpicture}
\caption{The lifts $\tilde{e}_1$ and $\tilde{e}_2$ of $e\in E_+(Y' - T)$ are identified by the hyperbolic translation $x_e$.  The axis for the action of $x_e$ is shown in blue, while a portion of $\tT$ is shown in red.}\label{fig:translate}
    
\end{figure}

Since $t$ acts trivially on $Y'$, for every point $x\in \tY$, there exists $g \in F_n$ such that $(gt).x=x$.  In particular, for each vertex $\tilde{v}\in \tT$, we can find $g_v\in F_n$ such that $(g_vt ).\tilde{v}=\tilde{v}$.  By the choice of $t$, this implies that $(g_v t)a_v(g_v t)^{-1} = a_v$ for all $v$, or in other words, that all vertex stabilizers are $\ZZ^2$. Moreover, replacing $t$ with $g_{v_0}t$, we may assume the stabilizer of $\tilde{v}_0$ is $\I{a_0,t}.$

Let $\gamma_v = (e_1,\ldots,e_m)$ be the minimal length edge path in $T$ from $v_0$ to $v$.  We claim that there exists integers $p_e$ such that for the word $w_v = a_{o(e_1)}^{p_{e_1}}a_{o(e_2)}^{p_{e_2}} \cdots a_{o(e_m)}^{p_{e_m}}$ we have $t.\tilde{v}=w_v.\tilde{v}$.  Note, the integers $p_e$ only depend on the edges and not the vertex $v$.  We prove this by induction on $m+1\geq 0$, where $w_{v_0}$ is understood to be trivial.  In the base case, we have already chosen $t$ so that $t.\tilde{v}_0=\tilde{v}_0$.  Suppose now that the claim holds for some $m\geq 0$, and that $\gamma_v$ has length $m+1$. Then removing the last edge of $\gamma_v$ is the path from $v$ to $v_m$, so there exist integers $p_e$ such that for the word $w_{v_m} =  a_{o(e_1)}^{p_{e_1}}a_{o(e_2)}^{p_{e_2}} \cdots a_{o(e_{m-1})}^{p_{e_{m-1}}}$ we have $t.\tilde{v}_{m}=w_{v_m}.\tilde{v}_{m}$. Let $\tilde{e}$ be the lift of $e_m$, which is the final edge of $\gamma_v$, to $\tT$. Since $w_{v_m}^{-1}t$ fixes $\tilde{v}_{m}$, there is some power $a_{v_m}^{-p_e}$ such that $a_{v_m}^{-p_e}(w_{v_m}^{-1}t).\tilde{e}=\tilde{e}$. Thus, $t.\tilde{e}=(w_{v_m}a_{v_m}^{p_e}).\tilde{e}$ and hence \[t.\tilde{v}=(w_{v_m}a_{o(e_m)}^{p_{e_m}}).\tilde{v}=(a_{o(e_1)}^{p_{e_1}}\cdots a_{o(e_{m-1})}^{p_{e_{m-1}}}a_{o(e_m)}^{p_{e_m}}.\tilde{v},\] as desired. It follows that for each $v$, $(w_v^{-1}t)a_v(t^{-1}w_v) = a_v$. Hence, $ta_vt^{-1}=w_v a_v w_v^{-1}$.  This proves that the $t$-action on the basis elements $a_v$ has the form indicated in Definition \ref{def:tubular aut}. 

Consider now an edge $e\in E_+(Y' - T)$ and let $\tilde{v}_1$ and $\tilde{v}_2$ be the lifts to $\tT$ of the vertices $v_{1} = o(e)$ and $v_2=\tau(e)$ respectively. There are lifts $\tilde{e}_1$ and $\tilde{e}_2$ of $e$ where $o(\tilde{e}_1) = \tilde{v}_{1}$ and $\tau(\tilde{e}_2) = \tilde{v}_{2}$ and such that $x_e . \tilde{e}_1 = \tilde{e}_2$.  Since $w_{v_1}^{-1}t$ fixes $\tilde{v}_{1}$, there exists a power $a_{v_1}^{-q_e}$ such that $(a_{v_1}^{-q_e}w_{v_1}^{-1}t) . \tilde{e}_1= \tilde{e}_1$. Similarly, there exists a power $a_{v_2}^{r_e}$ such that $(a_{v_2}^{r_e}w_{v_2}^{-1}t).\tilde{e}_2=\tilde{e}_2$. Then we obtain\[  \left( (t^{-1}w_{v_1}a_{v_1}^{q_e})x_e(a_{v_2}^{r_e}w_{v_2}^{-1}t)\right).\tilde{e}_1=x_e.\tilde{e}_2\]
Since edge stabilizers for the $F_n$--action are trivial we have that $x_e$ is the unique element of $F_n$ taking $\tilde{e}_1$ to $\tilde{e}_2$.  Hence, recalling that $v_1 = o(e)$ and $v_2 = \tau(e)$, we conclude that \[tx_et^{-1}=w_{o(e)}a_{o(e)}^{q_e} x_e a_{\tau(e)}^{r_e}w_{\tau(e)}^{-1}.\]

As explained in the beginning of the proof, this shows that some power of $\phi$ is geometrically linear unipotent. 
\end{proof}

\subsection{Proof of Theorem \ref{thm:FbyZ}}\label{subsec:FbyZ proof} 

There are two items left to show for Theorem~\ref{thm:FbyZ}.  We must show that if some power of an outer automorphism is geometrically linear unipotent (3), then $G_\phi$ is virtually tubular (2) and that if $G_\phi$ is virtually tubular 92), then $\omega(X) = 0$ for every $G_\phi$--complex $X$ (1).  Additionally, we must show that if none of these conditions holds that $\uomega(G_\phi)$ is bounded away from zero.  This last statement follows immediately from Proposition~\ref{prop:complementary} and Lemma~\ref{lem:uniform uniform FbyZ}.

\begin{proof}[Proof of Theorem~\ref{thm:FbyZ}]
First we show that (3) implies (2).  Replacing $\phi$ with a power replaces $G_\phi$ with a finite-index subgroup.  We will therefore assume $\phi$ is itself GLU, and prove that in this case $G_\phi$ is tubular. Let $\calF$, $T\subseteq Y'$, $v_0\in T$ and $\Phi$ be as in Definition \ref{def:tubular aut}.  These data give us a basis $\{a_v,x_e\}$, where $v\in VY'$ runs over the vertex set of $Y$, and $e\in E_+(Y' - T)$ runs over the edges in the complement of the maximal tree $T$. 

First we show that the subgroup $G_0=\I{t,a_v\mid v\in VY'}$ is tubular. For each $a_v$ we have $ta_vt^{-1}=\Phi(a_v)=w_va_vw_v^{-1}$.  Rearranging we obtain \[(w_v^{-1}t)a_v(t^{-1} w_v)a_v^{-1}=[w_v^{-1}t,a_v]=1,\]
or in other words $\I{a_v,w_v^{-1}t}\cong \ZZ^2$. Let $\gamma_v=e_1, \ldots, e_m$ be the minimal length edge path from $v_0$ to $v$ in $T$.  The condition that $\phi$ is GLU states that \[w_v=a_{o(e_1)}^{p_{e_1}}\cdots a_{o(e_m)}^{p_{e_m}}=w_{v'}a_{v'}^{p_{e_m}},\] for some collection of integers $p_e$ depending only on the edge $e$, and where $v'=o(e_m)$. In particular, $w_v^{-1}t\in \I{a_{v'},w_{v'}^{-1}t}$. We then build a graph of groups decomposition with graph $T$, vertex groups $\I{a_v,w_v^{-1}t}$ and edge groups $\ZZ$: the edge between $v$ and $v'$ identifies $\I{w_v^{-1}t}$ in each.  Thus $G_0$ is tubular.

To finish the proof, we show that adding the elements $x_e$ for each $ e\in E_+(Y' - T)$ exhibits $G_\phi$ as an iterated HNN-extension of $G_0$ over $\ZZ$ edge groups.  Each extension is independent for the others, so that the result is tubular.  For each element $x_e$, we have $tx_et^{-1}=\Phi(x_e)=w_{o(e)}a_{o(e)}^{q_e}x_ea_{\tau(e)}^{r_e}w_{\tau(e)}^{-1}$, or written differently:
\[x_e\left(a_{\tau(e)}^{r_e}w_{\tau(e)}^{-1}t\right)x_e^{-1}=w_{o(e)}^{-1}a_{o(e)}^{-q_e}t\]
Clearly, $a_{\tau(e)}^{r_e}w_{\tau(e)}^{-1}t\in \I{a_{\tau(e)},w_{\tau(e)}^{-1}t}$ while $w_{o(e)}^{-1}a_{o(e)}^{-q_e}t\in \I{a_{o(e},w_{o(e)}^{-1}t}$. We see that the addition of each edge $e\in E_+(Y' - T)$ adds a generator conjugating an infinite cyclic subgroups of the $\tau(e)$ vertex group to one of the $o(e)$ vertex group. Hence $G_\phi$ is tubular.  

Next, we observe that (2) implies (1).  If $G_\phi$ is virtually tubular, then by Lemma \ref{lem:comparison}, after passing to a tubular finite index subgroup, we have that $\omega(X) = 0$ for every $G_\phi$--complex by Theorem \ref{thm:zero iff graph of groups} as $\udelta(G_\phi) > 0$ (Lemma~\ref{lem:uniform uniform FbyZ}).  

Finally, by Proposition~\ref{prop:complementary} and Lemma~\ref{lem:uniform uniform FbyZ}, if $\omega(X) \neq 0$ for some $G_\phi$--complex $X$ then \[\uomega(G_\phi) \geq \frac{\udelta(G_\phi)}{2 \cdot \const} = \frac{\log 3}{12 \cdot \const}. \qedhere \]
\end{proof}

%%%%%%%%%%%%%%%%%%%%%%%%%%%%%%%%%%%%%%%%%%%%%%%%%%%%%%%%%%%%%%%%%%%%%%%%%%%%%

\section{2--Dimensional right-angled Artin groups}\label{sec:raags}

In this section we consider 2--dimensional right-angled Artin groups and prove Theorem~\ref{thm:raags}.  This theorem asserts the equivalence of the following three conditions for a 2--dimensional right-angled Artin group $A_\Gamma$.
\begin{enumerate}
\item $\omega(X) = 0$ for every $A_\Gamma$--complex $X$.
\item $A_\Gamma$ is a free product of tubular groups and a free group.
\item $\Gamma$ is a forest.
\end{enumerate}
The equivalence of (1) and (2) will follow from Theorem~\ref{thm:zero iff graph of groups}.  Indeed, that (1) is a consequence of (2) follows immediately from Theorem~\ref{thm:zero iff graph of groups} once we show that $\udelta(H) \geq \log 3$ for a nonabelian subgroup of a right-angled Artin group (Lemma~\ref{lem:uniform uniform raag}).  For the converse, we will show how to modify the graph of groups decomposition ensured by Theorem~\ref{thm:zero iff graph of groups} when $\omega(X) = 0$ for every $A_\Gamma$--complex $X$ in Lemma~\ref{lem:omega = 0 -> tubular} to get the requirements on the vertex groups.  In Lemmas~\ref{lem:tubular -> forest} and~\ref{lem:forest -> tubular} we establish the equivalence of (2) and (3).  Before we prove these lemmas, we recall a theorem of Baudisch regarding subgroups of a right-angled Artin group generated by two elements which is important for our analysis.

\subsection{Baudisch's theorem and consequences.}\label{subsec:baudisch}  
Baudisch proved that the subgroup generated by two elements in a right-angled Artin group is either free or abelian~\cite{ar:Baudisch81}.  We record two consequences of this fact.

\begin{lemma}\label{lem:2-subgroups of raags}
Groups in the follow classes are not isomorphic to a subgroup of a right-angled Artin group:
\begin{enumerate}
\item Baumslag--Solitar groups $BS(p,q) = \I{a,t \mid ta^pt^{-1} = a^q }$ if $p,q \neq 1$, 

\item torus knot groups $\ZZ \ast_{\ZZ} \ZZ = \I{x,y \mid x^p = y^q}$ if $p \neq 1$ or $q \neq 1$, and

\item amalgamated free products $\ZZ \ast_{\ZZ} \ZZ^2 = \I{x,y,z \mid x^p = w(y,z), \, yz = zy}$ where $w(y,z)$ is a nontrivial element of the subgroup $\I{y,z} \cong \ZZ^2$ if $p \neq 1$.
\end{enumerate}  
\end{lemma}

\begin{proof}
The first two classes are immediate since they are generated by two elements and are neither free nor abelian.  For the third class of groups, notice that either $w(y,z) \neq y$ or $w(y,z) \neq z$.  Hence one of the subgroups $\I{x,y}$ or $\I{x,z}$ is neither free nor abelian if $p \neq 1$.
\end{proof}

Another consequence of Baudisch's theorem is that right-angled Artin groups have uniform uniform exponential growth (with a uniform constant).

\begin{lemma}\label{lem:uniform uniform raag}
Let $A_\Gamma$ be a right-angled Artin group.  Suppose $H$ is a finitely generated nonabelian subgroup of $A_\Gamma$.  Then $\delta(H) \geq \log 3$.  In particular, if $A_\Gamma$ is nonabelian, then $\udelta(A_\Gamma) \geq \log 3$.
\end{lemma}

\begin{proof}
Suppose $A_\Gamma$ is a right-angled Artin group and that $H$ is a finitely generated nonabelian subgroup of $A_\Gamma$.  Let $S$ be a finite generating set for $H$.  As $H$ is nonabelian, there are two element $x,y \in S$ that do not commute.  Hence as $x$ and $y$ do not commute, they must generate a nonabelian free group by Baudisch's theorem and therefore $\delta(H,S) \geq \log 3$. 
\end{proof}

In particular, any subexponentially growing subgroup of a right-angled Artin group is abelian and hence right-angled Artin groups have Property $U$.

\subsection{Proof of Theorem~\ref{thm:raags}}\label{subsec:raags proof}
We can now prove the lemmas showing the equivalences of the three items in Theorem~\ref{thm:raags}.

\begin{lemma}\label{lem:omega = 0 -> tubular}
Suppose that $A_\Gamma$ is a right-angled Artin group with $\gdim(A_\Gamma) = 2$.  If $\omega(X) = 0$ for every $A_\Gamma$--complex $X$, then $A_\Gamma$ is a free product of tubular groups and a free group.
\end{lemma}

\begin{proof}
Let $A_\Gamma$ be a right-angled Artin group where $\gdim(A_\Gamma) = 2$ and let $X$ be an $A_\Gamma$--complex where $\omega(X) = 0$.  As $BS(1,-1)$ is not isomorphic to a subgroup of $A_\Gamma$ (Lemma~\ref{lem:2-subgroups of raags}) by Theorem~\ref{thm:zero iff graph of groups} we have that $A_\Gamma$ is the fundamental group of a graph of groups $\calG = (Y,\{G_v\},\{G_e\},\{h_e\})$ where the edge groups belong to the collection $\{\{1\},\ZZ\}$ and the vertex groups belong to the collection $\{\ZZ,\ZZ^2\}$.    

As in the proof of Theorem~\ref{thm:zero iff graph of groups}, we may assume that $\calG$ is reduced.  We will show that by altering $\calG$ that we can arrange that all of the vertex groups are $\ZZ^2$.  This will complete the proof of the lemma. 

To this end, suppose that there is a vertex $v$ in $Y$ with $G_v = \ZZ$.  Further suppose that there is an $e$ with $o(e) = \tau(e) = v$ and $G_e = \ZZ$.  Then vertex $v$ and edge $e$ correspond to an HNN-extension $\ZZ \ast_\ZZ \cong BS(p,q)$.  By Lemma~\ref{lem:2-subgroups of raags}, we must have $p=q=1$.  Hence, we can remove the edge $e$ and replace the vertex group $G_v$ with $\ZZ^2$.       

Next suppose that for every edge $e$ with $o(e) = v$ we have $G_e = \{1\}$.  In this case, we can add a new edge $e'$ that is a loop based at $v$ with group $G_{e'} = \{1\}$ and replace the vertex group $G_v$ with $\{1\}$.  The resulting graph of groups is not reduced.  Indeed, if it were then every edge incident on $v$ is loop which implies that $A_\Gamma$ is free.  Thus, there is some edge incident to $v$ that is collapsible.  When we collapse this edge we obtain a reduced graph of groups decomposition for $A_\Gamma$ which has fewer vertices with vertex group $\ZZ$ and has not created a vertex with vertex group $\{1\}$.   

Finally, we deal with the case that there are no loops at $v$ with edge group equal to $\ZZ$ and that there is an edge $e$ with $o(e) = v$ such that $G_e = \ZZ$.  Then the subgroup $G_{o(e)} \ast_{G_e} G_{\tau(e)}$ is either a torus knot group $\ZZ \ast_{\ZZ} \ZZ$ or an amalgamated free product $\ZZ \ast_{\ZZ} \ZZ^2$.  By Lemma~\ref{lem:2-subgroups of raags}, we must have that $h_e\from  G_e \to G_{o(e)}$ or $h_{\bar{e}} \from G_e \to G_{\tau(e)}$ is surjective, contradicting the assumption that the graph of groups decomposition is reduced.  
\end{proof}

\begin{lemma}\label{lem:tubular -> forest}
Suppose that $A_\Gamma$ is a right-angled Artin group with $\gdim(A_\Gamma) = 2$.  If $A_\Gamma$ is a free product of tubular groups and a free group, then $\Gamma$ is a forest.
\end{lemma}

\begin{proof}
Let $A_\Gamma$ be a 2--dimensional right-angled Artin group and suppose that $A_\Gamma \cong G_1 \ast \cdots \ast G_k \ast F_n$ where $G_1,\ldots,G_k$ are tubular groups and $F_n$ is a free group.  As tubular groups are one-ended, this decomposition must be the Grusko decomposition and thus we must have that $G_i \cong A_{\Gamma_i}$ for $i = 1,\ldots,k$ where $\Gamma_1,\ldots,\Gamma_k$ are the connected components of $\Gamma$ that have at least two vertices.  Therefore, it suffices to show that if $\Gamma$ is connected and $A_{\Gamma}$ is tubular, then $\Gamma$ is a tree.

To this end, suppose that we have a graph of groups decomposition of $A_{\Gamma}$ where each edge group is $\ZZ$ and each vertex groups is $\ZZ^2$.  We will use the calculation of the the JSJ--decomposition of a one-ended right-angled Artin group over cyclic subgroups by the second author~\cite[Section~3]{ar:Clay14}.  The vertex groups of the JSJ--decomposition of $A_{\Gamma}$ correspond to the biconnected components of $\Gamma$.  Specifially, vertex groups are of the form $A_{\Gamma_0}$ where $\Gamma_0 \subseteq \Gamma$ is a biconnected component.  The most important property of this decomposition is that the vertex groups are subgroups of conjugates of the vertex groups for any graph of groups decomposition of $A_{\Gamma}$ in which the edge groups are each equal to $\ZZ$~\cite[Lemma~3.3]{ar:Clay14}.  Hence $A_{\Gamma_0} = \ZZ^2$ for each biconnected component $\Gamma_0 \subseteq \Gamma$ and thus every biconnected component is a single edge.  Therefore every edge of $\Gamma$ is separating and hence $\Gamma$ is a tree.
\end{proof}

\begin{lemma}\label{lem:forest -> tubular}
Suppose that $A_\Gamma$ is a right-angled Artin group with $\gdim(A_\Gamma) = 2$.  If $\Gamma$ is a forest, then $A_\Gamma$ is the free product of tubular groups and a free group.
\end{lemma}

\begin{proof}
As in the proof of Lemma~\ref{lem:tubular -> forest}, it suffices to show that if $\Gamma$ is a nontrivial tree, then $A_{\Gamma}$ is tubular.  We do so by showing that $A_{\Gamma}$ is the fundamental group of a graph of groups $\calG = (Y,\{G_v\}, \{G_e\},\{h_e\})$ where all of edges groups are equal to $\ZZ$ and all of the vertex groups are equal to $\ZZ^2$.

Let $\widehat{Y}$ by the graph obtained by subdivision of $\Gamma$.  We will first construct a graph of groups $\widehat{\calG}$ with underlying graph $\widehat{Y}$ so that $\pi_1(\widehat{\calG}) = A_\Gamma$.  As $\widehat{Y}$ is the subdivision of $\Gamma$, we can consider $V\Gamma$ as a subset of $V\widehat{Y}$.  For each vertex $v \in V\Gamma$ we set $G_v = \ZZ$ and for each vertex $v \in V\widehat{Y} - V\Gamma$ we set $G_v = \ZZ^2$.  In the latter case, there are exactly two edges $e,e' \in E\widehat{Y}$ with $o(e) = o(e') = v$.  Decompose $G_e$ into the direct sum of two copies of $\ZZ$ denoted $\ZZ_e$ and $\ZZ_{e'}$ respectively so that $G_v = \ZZ_e \oplus \ZZ_{e'}$.  For each edge $e \in E\widehat{Y}$ we set $G_e = \ZZ$.  The inclusion maps $h_e \from G_e \to G_{o(e)}$ are defined as follows.  If $o(e)$ lies in $V\Gamma$, then $G_{o(e)} = \ZZ$ and we define $h_e \from G_e \to G_{o(e)}$ to be an isomorphism.  Else we have that $G_{o(e)} = \ZZ^2 = \ZZ_e \oplus \ZZ_{e'}$ and we define $h_E \from G_e \to G_{o(e)}$ to have $\ZZ_e$ as image.  We set $\widehat{\calG}$ to be the graph of groups $(\widehat{Y},\{G_v\},\{G_e\},\{h_e\})$.    

As $\Gamma$ is a tree, the presentation for $\pi_1(\widehat{\calG})$ shows that $\pi_1(\widehat{\calG}) \cong A_\Gamma$.  

Notice that $\widehat{\calG}$ is not reduced.  Indeed, for every vertex in $V\Gamma$, there is an incident edge that is not a loop and for which the inclusion map is an isomorphism.  For each vertex in $V\Gamma$, we fix one such edge and perform the collapse.  Let $\calG = (Y,\{G_v\},\{G_e\},\{h_e\})$ be the resulting graph of groups.  We then observe that the vertex set $VY$ corresponds the set set $V\widehat{Y} - V\Gamma$ and hence $G_v = \ZZ^2$ for each vertex $v \in VY$.  The edge groups do not change and hence $G_e = \ZZ$ for each edge $e \in EY$.  As $\pi_1(\calG) = \pi_1(\widehat{\calG}) \cong A_\Gamma$, we have shown that $A_\Gamma$ is tubular.   
\end{proof}

\begin{proof}[Proof of Theorem~\ref{thm:raags}]
Let $A_\Gamma$ be a right-angled Artin group where $\gdim(A_\Gamma) = 2$.

The equivalences of items (1) and (2) follow from Theorem~\ref{thm:zero iff graph of groups} and Lemmas~\ref{lem:uniform uniform raag} and~\ref{lem:omega = 0 -> tubular}.  The equivalences of items (2) and (3) follow immediately from Lemmas~\ref{lem:tubular -> forest} and~\ref{lem:forest -> tubular}.

Finally, by Proposition~\ref{prop:complementary} and Lemma~\ref{lem:uniform uniform raag}, if $\omega(X) \neq 0$ for some $A_\Gamma$--complex $X$ then \[\uomega(A_\Gamma) \geq \frac{\udelta(A_\Gamma)}{2 \cdot \const} = \frac{\log 3}{2 \cdot \const}. \qedhere \]
\end{proof}

%%%%%%%%%%%%%%%%%%%%%%%%%%%%%%%%%%%%%%%%%%%%%%%%%%%%%%%%%%%%%%%%%%%%%%%%%%%%%

\bibliography{mve}
\bibliographystyle{acm}

\end{document}